\documentclass[12pt,reqno,a4paper]{amsart}
\usepackage{fullpage}
\usepackage[utf8]{inputenc}
\usepackage[T1]{fontenc}
\usepackage{hyperref} 
\usepackage{enumitem}
\usepackage{amsmath,amssymb,amsthm,color}
\usepackage{todonotes,framed}
\usepackage{nameref}
\usepackage{enumitem}

\title{Rate optimality of adaptive finite element methods\\with respect to overall computational costs}
\author{Gregor Gantner}
\author{Alexander Haberl}
\author{Dirk Praetorius}
\author{Stefan Schimanko}

\address{University of Amsterdam, Korteweg-De Vries Institute for Mathematics, P.O.\ Box 94248, 1090 GE Amsterdam, The Netherlands}
\email{g.gantner@uva.nl}
 
\address{TU Wien, Institute for Analysis and Scientific Computing, Wiedner Hauptstr.\ 8-10/E101/4, 1040 Wien, Austria}
\email{alexander.haberl@asc.tuwien.ac.at}
\email{dirk.praetorius@asc.tuwien.ac.at\qquad\rm(corresponding author)}
\email{stefan.schimanko@asc.tuwien.ac.at}

\subjclass[2010]{65N30, 65N50, 65Y20, 65N22, 41A25.}
\thanks{{\bf Acknowledgement.} The authors thankfully acknowledge support by the Austrian Science Fund (FWF) through the doctoral school \emph{Dissipation and dispersion in nonlinear PDEs} (grant W1245), the SFB \emph{Taming complexity in partial differential systems} (grant SFB F65), and the stand-alone projects \emph{Optimal isogeometric boundary element methods} (grant P29096), \emph{Computational nonlinear PDEs} (grant P33216), and \emph{Optimal adaptivity for space-time methods} (grant J4379-N)}

\makeatother

\def\set#1#2{\big\{#1 \,:\, #2 \big\}}
\def\QQ{\mathcal{Q}}
\def\Alpha{{\rm A}}
\def\dist{{\rm d\!l}}
\def\eps{\varepsilon}

\usepackage{fancyhdr}
\advance\footskip0.4cm
\advance\textheight-0.4cm
\calclayout
\newcommand*\patchAmsMathEnvironmentForLineno[1]{%
  \expandafter\let\csname old#1\expandafter\endcsname\csname #1\endcsname
  \expandafter\let\csname oldend#1\expandafter\endcsname\csname end#1\endcsname
  \renewenvironment{#1}%
     {\linenomath\csname old#1\endcsname}%
     {\csname oldend#1\endcsname\endlinenomath}}%
\newcommand*\patchBothAmsMathEnvironmentsForLineno[1]{%
  \patchAmsMathEnvironmentForLineno{#1}%
  \patchAmsMathEnvironmentForLineno{#1*}}%
\AtBeginDocument{%
\patchBothAmsMathEnvironmentsForLineno{equation}%
\patchBothAmsMathEnvironmentsForLineno{align}%
\patchBothAmsMathEnvironmentsForLineno{flalign}%
\patchBothAmsMathEnvironmentsForLineno{alignat}%
\patchBothAmsMathEnvironmentsForLineno{gather}%
\patchBothAmsMathEnvironmentsForLineno{multline}%
}
\usepackage[mathlines]{lineno}

\makeatletter
\def\@seccntformat#1{\hspace*{4mm}%
  \protect\textup{\protect\@secnumfont
    \ifnum\pdfstrcmp{subsection}{#1}=0 \bfseries\fi
    \csname the#1\endcsname
    \protect\@secnumpunct
  }%
}
\makeatother

\makeatletter
\def\section{\@startsection{section}{1}%
\z@{.7\linespacing\@plus\linespacing}{.5\linespacing}%
{\normalsize\scshape\bfseries\centering}}
\makeatother

\makeatletter
\renewcommand{\@secnumfont}{\bfseries}
\makeatother

\def\coarse{H}
\def\fine{h}

\def\N{\mathbb{N}}
\def\T{\mathbb{T}}
\def\MM{\mathcal{M}}
\def\TT{\mathcal{T}}
\def\XX{\mathcal{X}}

\def\k{{\underline{k}}}

\def\Crel{C_{\rm rel}}
\def\Cpcg{C_{\rm pcg}}
\def\qpcg{q_{\rm pcg}}
\def\Cstab{C_{\rm stab}}
\def\qred{q_{\rm red}}
\def\Clin{C_{\rm lin}}
\def\qlin{q_{\rm lin}}
\def\qctr{q_{\rm ctr}}
\def\Ccea{C_{\text{\rm C\'ea}}}

\def\refine{{\tt refine}}

\def\enorm#1{|\!|\!| #1 |\!|\!|}
\def\reff#1#2{\stackrel{\eqref{#1}}{#2}}

\def\dual#1#2{\langle#1\,,\,#2\rangle}

\def\Cmark{C_{\rm mark}}
\def\Copt{C_{\rm opt}}
\def\norm#1#2{\|#1\|_{#2}}
\def\A{\mathbb{A}}
\def\OO{\mathcal{O}}
\def\R{\mathbb{R}}
\def\d#1{\,{\rm d}#1}
\def\div{{\rm div}\,}
\def\lctr{\lambda_{\rm ctr}}
\def\HH{\mathcal{H}}

\def\AA{\mathcal{A}}
\def\HH{\mathcal{H}}

\def\EE{\mathcal{E}}
\def\K{\mathbb{K}}
\def\C{\mathbb{C}}
\def\product#1#2{\langle\!\langle #1,#2 \rangle\!\rangle}
\def\QQ{\mathcal{Q}}
\def\UU{\mathcal{U}}

\def\lconv{\lambda_{\rm conv}}

\def\Cghps{C_{\rm ghps}}
\def\qghps{q_{\rm ghps}}
\def\Csum{C_{\rm sum}}

\def\RR{\mathcal{R}}
\def\opt{{\rm opt}}
\def\lopt{\lambda_{\rm opt}}
\def\copt{c_{\rm opt}}

\newcounter{statement}
\newenvironment{statement}[2][!]{%
\vskip3mm
\hrule
\hrule
\hrule
\vskip1mm
\noindent%
\refstepcounter{statement}%
\bf#2~\thestatement%
\ifthenelse{\equal{#1}{!}}{.\ }{~(#1).\ }%
\it%
}{%
\vskip1mm
\hrule
\hrule
\hrule
\vskip2mm
}

\newenvironment{theorem}[1][!]{\begin{statement}[#1]{Theorem}}{\end{statement}}
\newenvironment{lemma}[1][!]{\begin{statement}[#1]{Lemma}}{\end{statement}}

\newenvironment{proposition}[1][!]{\begin{statement}[#1]{Proposition}}{\end{statement}}
\newenvironment{corollary}[1][!]{\begin{statement}[#1]{Corollary}}{\end{statement}}
\newenvironment{remark}[1][!]{\begin{statement}[#1]{Remark}}{\end{statement}}
\newenvironment{algorithm}[1][!]{\begin{statement}[#1]{Algorithm}}{\end{statement}}

\begin{document}

\begin{abstract}
We consider adaptive finite element methods for second-order elliptic PDEs, where the arising discrete systems are not solved exactly. For contractive iterative solvers, we formulate an adaptive algorithm which monitors and steers the adaptive mesh-refinement as well as the inexact solution of the arising discrete systems. We prove that the proposed strategy leads to linear convergence with optimal algebraic rates. Unlike prior works, however, we focus on convergence rates with respect to the overall computational costs. In explicit terms, the proposed adaptive strategy thus guarantees quasi-optimal computational time. In particular, our analysis covers linear problems, where the linear systems are solved by an optimally preconditioned CG method as well as nonlinear problems with strongly monotone nonlinearity which are linearized by the so-called Zarantonello iteration.
\end{abstract}

\maketitle
\thispagestyle{fancy}

\section{Introduction}
\label{section:introduction}

\subsection{State of the art}

The ultimate goal of any numerical scheme is to compute a discrete solution with error below a prescribed tolerance at, up to a multiplicative constant, the minimal computational costs. Since the convergence of numerical methods is usually spoiled by singularities of the (given) data as well as the (unknown) solution, {\sl a~posteriori} error estimation and related adaptive mesh-refinement strategies are indispensable tools for reliable numerical simulations. For many model problems, the mathematical understanding of rate-optimal convergence of adaptive FEM has matured; we refer to~\cite{doerfler1996,mns2000,bdd2004,stevenson2007,ckns2008,cn2012,ffp2014} for some works for linear problems, to~\cite{veeser2002,dk2008,bdk2012,gmz2012} for nonlinear problems, and to~\cite{axioms} for a general framework of convergence of adaptive FEM with optimal convergence rates. 
Some works also account for the approximate computation of the discrete solutions by iterative (and inexact) solvers; see, e.g., \cite{bms2010,agl2013} for linear problems and~\cite{gmz2011,banach,hw18,hw19} for nonlinear model problems.
Moreover, there are many papers on {\sl a~posteriori} error estimation which also include the iterative and inexact solution for nonlinear problems; see, e.g.,~\cite{eev2011,ev2013,aw2015,hw2018} and the references therein.

As far as optimal convergence rates are concerned, the mentioned works focus on rates with respect to the degrees of freedom. However, in practice, one aims for the optimal rate of convergence with respect to the computational costs, i.e., the computational time. The issue of optimal computational costs is already addressed in the seminal work~\cite{stevenson2007} for the Poisson model problem. There, it is assumed that a sufficiently accurate discrete solution can be computed in linear complexity, e.g., by a multigrid solver. Under these so-called \emph{realistic assumptions} on the solver, it is then proved that the \emph{total error} (i.e., the sum of energy error plus data oscillations) will also converge with optimal rate with respect to the computational costs.
A similar result is obtained in~\cite{cg2012} for an adaptive Laplace eigenvalue computation.

In recent own works, we have tried to include concrete solvers into the convergence analysis. In~\cite{banach}, we considered adaptive FEM for an elliptic PDE with strongly monotone nonlinearity. The arising nonlinear FEM problems are linearized via the so-called \emph{Zarantonello iteration} (or \emph{Banach--Picard iteration}), which leads to a \emph{linear} Poisson problem in each step. The adaptive algorithm drives the linearization strategy as well as the local mesh-refinement. In~\cite{banach}, we have proved that the overall strategy leads to optimal convergence rates with respect to the degrees of freedom and to \emph{almost optimal} convergence rates with respect to the total computational costs. The latter means that, if the total error converges with rate $s > 0$ with respect to the degrees of freedom, then it converges with rate $s - \eps > 0$ with respect to the overall computational costs, for all $\eps > 0$. Moreover, in~\cite{abem+solve}, we obtained analogous results for an adaptive boundary element method, where we employed a preconditioned conjugate gradient method (PCG) with optimal additive Schwarz preconditioner to approximately solve the arising linear discrete systems.

In the present work, we now prove \emph{optimal rates with respect to the overall computational costs} for the algorithm from~\cite{banach}. Moreover, we give an abstract analysis in the spirit of~\cite{axioms} and show that this also covers linear solvers like PCG. The precise contributions of this work are outlined in the remaining subsections of this introduction.

\subsection{Model problem}

We consider the elliptic boundary value problem
\begin{align}\label{eq:strong}
 \begin{split}
 -\div A(\nabla u^\star) &= f \quad \text{in } \Omega,\\
 u^\star &= 0 \quad \text{on } \Gamma := \partial\Omega,
 \end{split}
\end{align}
where $\Omega \subset \R^d$ is a bounded Lipschitz domain with $d\in\{2,3\}$, and $f \in L^2(\Omega)$ is a given load. We assume that $A : L^2(\Omega)^d \to L^2(\Omega)^d$ is strongly monotone and Lipschitz continuous (see Section~\ref{section:abstract}) so that~\eqref{eq:strong} resp.\ the equivalent variational formulation
\begin{align}\label{eq:weak}
\dual{\AA u^\star}{v}_{\HH'\times\HH}:= \int_\Omega A(\nabla u^\star) \cdot \nabla v \d{x}
= \int_\Omega f v_\ell \d{x}=:\dual{F}{v}_{\HH\times\HH'}\quad\text{for all }v\in\HH
\end{align}
  admit a unique solution $u^\star \in \HH := H^1_0(\Omega)$. Given a discrete subspace $\XX_\ell \subset \HH$ related to some triangulation $\TT_\ell$ of $\Omega$, also the discrete formulation
\begin{align}\label{eq:discrete}
\dual{\AA u_\ell^\star}{v_\ell}_{\HH\times\HH'}
 =\dual{F}{v_\ell}_{\HH'\times\HH}
 \quad \text{for all } v_\ell \in \XX_\ell
\end{align}
admits a unique solution $u_\ell^\star \in \XX_\ell$. If $A$ is nonlinear, then $u_\ell^\star$ can hardly be computed exactly. Even if $A$ is linear, usual FEM codes employ iterative solvers like PCG, GMRES, or multigrid. 

Given an initial guess $u_\ell^0 \in \XX_\ell$, we assume that we can compute iterates $u_\ell^k := \Phi_\ell(u_\ell^{k-1}) \in \XX_\ell$ which lead to a contraction in the energy norm on $\HH$, i.e.,
\begin{align}\label{eq:contraction}
 \enorm{u_\ell^\star - u_\ell^k}
 \le q \, \enorm{u_\ell^\star - u_\ell^{k-1}}
 \quad \text{for all } k \in \N
\end{align}
with some $\XX_\ell$-independent contraction constant $0 < q < 1$. In explicit terms, we assume that we have an iterative solver with iteration function $\Phi_\ell : \XX_\ell \to \XX_\ell$ which is uniformly contractive in each step. We assume that we can control the discretization error (for the exact, but never computed discrete solution $u_\ell^\star \in \XX_\ell$ from~\eqref{eq:discrete}) by some reliable {\sl a~posteriori} error estimator
\begin{align}
 \Crel^{-1} \, \enorm{u^\star - u_\ell^\star} \le \eta_\ell(u_\ell^\star) := \bigg(\sum_{T \in \TT_\ell} \eta_\ell(T,u_\ell^\star)^2 \bigg)^{1/2},
\end{align}
where the local indicators $\eta_\ell(T,\cdot)$ can also be evaluated for other discrete functions $v_\ell \in \XX_\ell$ instead of the exact Galerkin solution $u_\ell^\star \in \XX_\ell$.

\subsection{Adaptive algorithm}

In the following, we formulate our adaptive algorithm, which takes the form
\begin{align}
 \boxed{\text{ Iteratively Solve \& Estimate }}
 \quad \longrightarrow \quad
 \boxed{\text{ Mark\phantom{y}\!\!}}
 \quad \longrightarrow \quad
 \boxed{\text{ Refine\phantom{y}\!\!}}
\end{align}
where the first step may be understood (and stated) as an inner loop, and \emph{Mark} is based on the D\"orfler criterion~\eqref{eq:intro:doerfler} from~\cite{doerfler1996} with (quasi-) minimal cardinality~\cite{stevenson2007}.

Let $0 < \theta \le 1$, $\Cmark \ge 1$, and $\lambda_{\rm ctr} > 0$ be given. Starting from an initial mesh $\TT_0=\TT_1$ and an initial guess $u_0 \in \XX_0$, our adaptive algorithm iterates the following steps for all $n = 1, 2, 3, \dots$:
\begin{itemize}
\item[\rm(i)] Perform one step of the iterative solver to obtain $u_n := \Phi_n(u_{n-1})$.
\item[\rm(ii)] Compute the refinement indicators $\eta_n(T,u_n)$ for all $T \in \TT_n$.
\item[\rm(iii)] If $\enorm{u_n - u_{n-1}} > \lambda_{\rm ctr} \, \eta_n(u_n)$, then define $\TT_{n+1} := \TT_n$ and continue with Step~{\rm(i)}.
\item[\rm(iv)] Otherwise, choose a set of marked elements $\MM_n \subseteq \TT_n$ which has up to the multiplicative constant $\Cmark$ minimal cardinality and satisfies the D\"orfler marking
\begin{align}\label{eq:intro:doerfler}
 \theta \, \bigg( \sum_{T \in \TT_n} \eta_n(T, u_n)^2 \bigg)^{1/2}
 \le \bigg( \sum_{T \in \MM_n} \eta_n(T, u_n)^2 \bigg)^{1/2} .
\end{align}
\item[\rm(v)] Employ, e.g., newest vertex bisection to generate the coarsest refinement $\TT_{n+1}$ of $\TT_n$ such that at least all marked elements are refined.
\end{itemize}
Note that the index $n \in \N_0$ does not distinguish whether one step of the iterative solver is performed or the mesh is locally refined.

We remark  that the computation of, e.g., all residual error indicators in Step~(ii) as well as as the local mesh-refinement by, e.g., newest vertex bisection can be done at linear costs. The same applies to, e.g., one step of PCG with an optimal additive Schwarz preconditioner in Step~(i). For the D\"orfler marking~\eqref{eq:intro:doerfler} in Step~(iv), we refer to~\cite{stevenson2007} for an algorithm with linear costs and $\Cmark = 2$ as well as to the recent own algorithm~\cite{pp2019} with linear costs and even $\Cmark = 1$.

\subsection{Main results}

Under usual assumptions, we prove that the proposed adaptive algorithm guarantees linear convergence of the \emph{quasi-error} (consisting of error plus error estimator) in the sense of
\begin{align}\label{eq:linear}
 \big( \enorm{u^\star - u_{n+m}} + \eta_{n+m}(u_{n+m}) \big)
 \le \Clin \qlin^n \, \big( \enorm{u^\star - u_m} + \eta_{m}(u_m) \big)
 \quad \text{for all } m,n \in \N_0;
\end{align}
see Theorem~\ref{theorem:linconv} for the precise statement. Moreover, given $N \in \N_0$, let $\T(N)$ be the set of all refinements $\TT$ of $\TT_0$ with $\#\TT - \#\TT_0 \le N$. Then, the algorithm leads to optimal convergence behavior in the following sense: For given $s > 0$, define
\begin{align}\label{eq:As}
 \norm{u^\star}{\A_s} := \sup_{N \in \N_0} (N+1)^s \inf_{\TT_{\rm opt} \in \T(N)} \big( \enorm{u^\star - u_{\rm opt}^\star} + \eta_{\rm opt}(u_{\rm opt}^\star) \big)
 \in \R_{\ge0} \cup \{\infty\}.
\end{align}
Then, there exists a constant $C(s) > 0$ such that
\begin{align}\label{eq:optimal intro}
 \begin{split}
C(s)^{-1} \, \norm{u^\star}{\A_s} 
 &\le \sup_{n \in \N_0} (\#\TT_n-\#\TT_0+1)^s \, \big( \enorm{u^\star - u_n} + \eta_n(u_n) \big)
 \\&
 \le \sup_{n \in \N_0} \bigg( \sum_{m = 0}^n \#\TT_m \bigg)^s \big( \enorm{u^\star - u_n} + \eta_n(u_n) \big)
 \le C(s) \, (\norm{u^\star}{\A_s}+1);
 \end{split}
\end{align}
see Theorem~\ref{theorem:rates} for the precise statement. Some comments are in order to underline the importance of the latter result:

{\large$\color{gray}\bullet$} 
First, by definition~\eqref{eq:As}, it holds that $\norm{u^\star}{\A_s} < \infty$ if and only if the quasi-error (for the exact discrete solutions) converges at least with algebraic rate $s > 0$ along a sequence of optimal meshes. 

{\large$\color{gray}\bullet$} 
Second, if all Steps~{\rm(i)--(v)} of the adaptive algorithm can be performed at linear costs $\OO(\#\TT_n)$, then the sum $\sum_{m = 0}^n \#\TT_m$ is proportional to the overall computational work (resp.\ the overall computational time spent) to perform the $n$-th step of the adaptive loop, since each adaptive step depends on the full adaptive history. 

{\large$\color{gray}\bullet$} 
Third, the interpretation of~\eqref{eq:optimal intro} thus is that the quasi-error for the computed discrete solutions $u_n$ decays with rate $s$ with respect to the overall computational costs (as well as the degrees of freedom) if and only if rate $s$ is possible with respect to the degrees of freedom (for the exact discrete solutions on optimal meshes).

{\large$\color{gray}\bullet$} 
Fourth, since $s > 0$ is arbitrary, the proposed algorithm will asymptotically regain the best possible convergence behavior, even with respect to the computational costs.

{\large$\color{gray}\bullet$} 
Prior works (see, e.g.,~\cite{stevenson2007,bms2010,cg2012,banach}) proved linear convergence of the quasi-error only for those steps, where mesh-refinement takes place. Unlike this, we prove linear convergence~\eqref{eq:linear} for the full sequence of discrete approximations, i.e., independently of the algorithmic decision for mesh-refinement or one step of the discrete solver. Moreover, our proof of~\eqref{eq:optimal intro} shows that \emph{full} linear convergence~\eqref{eq:linear} is the key argument to prove optimal rates with respect to the computational costs. 

{\large$\color{gray}\bullet$} 
In usual applications, the \emph{quasi-error} (i.e., error plus estimator) is equivalent to the so-called \emph{total error} (i.e., error plus data oscillations) as well as to the estimator alone. Therefore, the approximability $\norm{u^\star}{\A_s}$ in~\eqref{eq:As} can equivalently be defined through the total error (see, e.g., \cite{stevenson2007,ckns2008,cn2012,ffp2014}) or the estimator (see, e.g.,~\cite{axioms}) instead of the quasi-error (used in~\eqref{eq:As}). The overall result will be the same. 

\subsection{Outline}

The remainder of this work is organized as follows: First, Section~\ref{section:main_results} formulates the precise assumptions on the model problem (Section~\ref{section:abstract}), the mesh-refine\-ment and the FEM spaces (Section~\ref{section:mesh-refinement}--\ref{section:discrete-spaces}), and the error estimator and the iterative solver (Section~\ref{section:estimator}--\ref{section:solver}). Then, we give a reformulation of the proposed adaptive algorithm in Section~\ref{section:algorithm}, which is more appropriate for the numerical analysis, and formulate our main results in Section~\ref{section:main-results}. Some remarks on the abstract setting are collected in Section~\ref{section:remarks}, before we apply the setting to adaptive FEM with PCG solver for linear PDEs (Section~\ref{section:linear}) and the adaptive algorithm from~\cite{banach} for adaptive FEM for problems with strongly monotone nonlinearity (Section~\ref{section:nonlinear}). Section~\ref{section:reliable} provides the proof of Proposition~\ref{lemma:reliable} that the proposed adaptive strategy inherently allows for reliable {\sl a~posteriori} error control. The proof of Theorem~\ref{theorem:linconv} (full linear convergence) is given in Section~\ref{section:proofs}.  The proof of Theorem~\ref{section:costs} (optimal rates with respect to the computational costs) is given in Section~\ref{section:costs}. Some numerical experiments in Section~\ref{section:numerics} underline our theoretical findings by numerical experiments in 2D and conclude the work.

\section{Main results}
\label{section:main_results}

\subsection{Abstract model problem}
\label{section:abstract}

Let $\HH$ be a Hilbert space over $\K \in \{ \R, \C \}$ with scalar product $\product\cdot\cdot$, corresponding norm $\enorm\cdot$, and dual space $\HH'$ (with norm $\enorm{\cdot}'$).
Let $P: \HH \to \K$ be G\^ateaux differentiable with derivative $\AA := {\rm d}P: \HH \to \HH'$, i.e.,
\begin{align}\label{eq:potential}
 \dual{\AA w}{v}_{\HH'\times\HH} = \lim_{\substack{t \to 0 \\ t \in \R}} \frac{P(w+tv)-P(w)}{t}
 \quad \text{for all } v, w \in \HH.
\end{align}
We suppose that the operator $\AA$ is strongly monotone and Lipschitz continuous, i.e.,
\begin{align}\label{eq:monotone+lipschitz}
 \alpha \, \enorm{w - v}^2 \le {\rm Re} \, \dual{\AA w - \AA v}{w - v}_{\HH'\times\HH}
 \quad \text{and} \quad
 \enorm{\AA w - \AA v}' \le L \, \enorm{w - v}
\end{align}
for all $v, w \in \HH$, where $0 < \alpha \le L$ are given constants. For a linear and continuous functional $F \in \HH'$ and any closed subspace $\XX_\coarse \subseteq \HH$, the main theorem on monotone operators~\cite[Section~25.4]{zeidler} yields existence and uniqueness of the solution $u_\coarse^\star \in \XX_\coarse$ of
\begin{align}
 \dual{\AA u_\coarse^\star}{v_\coarse}_{\HH'\times\HH} = \dual{F}{v_\coarse}_{\HH'\times\HH}
 \quad \text{for all } v_\coarse \in \XX_\coarse. 
\end{align}
In particular, let $u^\star \in \HH$ denote the exact solution on $\HH$.
Moreover, with the energy functional $\EE := {\rm Re}\,(P - F)$, it holds that
\begin{align}\label{eq:energy}
 \frac{\alpha}{2} \, \enorm{u_\coarse^\star - v_\coarse}^2
 \le \EE(v_\coarse) - \EE(u_\coarse^\star) 
 \le \frac{L}{2} \, \enorm{u_\coarse^\star - v_\coarse}^2
 \quad \text{for all } v_\coarse \in \XX_\coarse;
\end{align}
see, e.g.,~\cite[Lemma~5.1]{banach}.
In particular, $u^\star \in \HH$ (resp.\ $u_\coarse^\star \in \XX_\coarse^\star$) is the unique minimizer of the minimization problem
\begin{align}\label{eq:minimization}
 \EE(u^\star) = \min_{v \in \HH} \EE(v)
 \quad \big( \text{resp.} \quad
 \EE(u_\coarse^\star) = \min_{v_\coarse \in \XX_\coarse} \EE(v_\coarse) \big).
\end{align}
As for linear elliptic problems, the present setting guarantees the C\'ea lemma
\begin{align}\label{eq:cea}
 \enorm{u^\star - u_\coarse^\star}
 \le \Ccea \, \enorm{u^\star - v_\coarse}
 \quad \text{for all } v_\coarse \in \XX_\coarse
 \quad \text{with} \quad
 \Ccea := L/\alpha.
\end{align}

\def\Cson{C_{\rm son}}
\def\Cmesh{C_{\rm mesh}}
\subsection{Mesh-refinement}
\label{section:mesh-refinement}

We assume that $\refine(\cdot)$ is a fixed mesh-refinement strategy, e.g., newest vertex bisection~\cite{stevenson2008}.
We write $\TT_\fine = \refine(\TT_\coarse,\MM_\coarse)$ for the coarsest one-level refinement of $\TT_\coarse$, where all marked elements $\MM_\coarse \subseteq \TT_\coarse$ have been refined, i.e., $\MM_\coarse \subseteq \TT_\coarse \backslash \TT_\fine$. We write $\TT_\fine \in \refine(\TT_\coarse)$, if $\TT_\fine$ can be obtained by finitely many steps of one-level refinement (with appropriate, yet arbitrary marked elements in each step). We define $\T := \refine(\TT_0)$ as the set of all meshes which can be generated from the initial mesh $\TT_0$ by use of $\refine(\cdot)$. 
Finally, we associate to each $\TT_\coarse\in\T$ a corresponding finite-dimensional subspace $\XX_\coarse$.

For our analysis, we only employ the following structural properties~\eqref{axiom:sons}--\eqref{axiom:closure}, where $\Cson \ge 2$ and $\Cmesh > 0$ are generic constants:
\renewcommand{\theenumi}{R\arabic{enumi}}
\begin{enumerate}
\bf
\item\label{axiom:sons}
\rm
\textbf{splitting property:} Each refined element is split into finitely many sons, i.e., for all $\TT_\coarse \in \mathbb{T}$ and 
all $\mathcal{M}_\coarse \subseteq \TT_\coarse$, the mesh
$\TT_\fine = \refine(\TT_\coarse, \mathcal{M}_\coarse)$ satisfies that
\begin{align*}
	\# (\TT_\coarse \setminus \TT_\fine) + \# \TT_\coarse \leq \# \TT_\fine
	\leq \Cson \, \# (\TT_\coarse \setminus \TT_\fine) + \# (\TT_\coarse \cap \TT_\fine).
\end{align*} 
\bf
\item\label{axiom:overlay} 
\rm
\textbf{overlay estimate:} For all meshes $\TT \in \mathbb{T}$ and $\TT_\coarse,\TT_\fine \in \refine(\TT)$,
there exists a common refinement $\TT_\coarse \oplus \TT_\fine \in \refine(\TT_\coarse) \cap \refine(\TT_\fine) \subseteq \refine(\TT)$ such that
\begin{align*}
	\# (\TT_\coarse \oplus \TT_\fine) \leq \# \TT_\coarse + \# \TT_\fine - \# \TT.
\end{align*}
\bf
\item\label{axiom:closure} 
\rm
\textbf{mesh-closure estimate:}
 For each sequence $(\TT_\ell)_{\ell \in \N_0}$ of successively refined meshes, i.e., $\TT_{\ell+1} := \refine(\TT_\ell,\MM_\ell)$ with $\MM_\ell \subseteq \TT_\ell$ for all $\ell \in \N_0$, it holds  that	
\begin{align*}
	\# \TT_\ell - \# \TT_0 \leq \Cmesh \sum_{j=0}^{\ell -1} \# \mathcal{M}_j.
\end{align*}%
\end{enumerate}

\subsection{Conforming discrete subspaces}
\label{section:discrete-spaces}

To each $\TT_\coarse\in\T$, we associate a finite-dimensional conforming $\XX_\coarse\subset \HH$. We suppose nestedness in the sense that 
$\XX_\coarse\subseteq\XX_\fine$ for all $\TT_\fine\in\refine(\TT_\coarse)$.

\subsection{Error estimator}
\label{section:estimator}

For each mesh $\TT_\coarse \in \T$, suppose that we can compute refinement indicators 
\begin{align}
 \eta_\coarse(T,v_\coarse) \ge 0
 \quad \text{for all } T \in \TT_\coarse
 \text{ and all } v_\coarse \in \XX_\coarse.
\end{align}
To abbreviate notation, let $\eta_\coarse(v_\coarse) := \eta_\coarse(\TT_\coarse, v_\coarse)$, where 
\begin{align}
 \eta_\coarse(\UU_\coarse, v_\coarse) 
 := \bigg(\sum_{T \in \UU_\coarse} \eta_\coarse(T,v_\coarse)^2 \bigg)^{1/2}
 \quad \text{for all } \UU_\coarse \subseteq \TT_\coarse.
\end{align}
We assume the following \emph{axioms of adaptivity} from~\cite{axioms} for all $\TT_\coarse \in \T$ and all $\TT_\fine \in \refine(\TT_\coarse)$, where $\Cstab, \Crel > 0$ and $0 < \qred < 1$ are generic constants:
\renewcommand{\theenumi}{A\arabic{enumi}}
\begin{enumerate}
\bf
\item stability:\label{axiom:stability}\rm\,
$| \eta_\fine(\UU_\coarse, v_\fine) - \eta_\coarse(\UU_\coarse,w_\coarse) | 
\le \Cstab \enorm{v_\fine - w_\coarse}$
 for all $v_\fine\in\XX_\fine, w_\coarse \in \XX_\coarse$ and all $\UU_\coarse \subseteq \TT_\coarse \cap \TT_\fine$.
\bf
\item reduction:\label{axiom:reduction}\rm\,
 $\eta_\fine(\TT_\fine \backslash \TT_\coarse, v_\coarse) \le \qred \, \eta_\coarse(\TT_\coarse \backslash \TT_\fine, v_\coarse)$ for all $v_\coarse \in \XX_\coarse$.
\bf
\item reliability:\label{axiom:reliability}\rm\, 
$\enorm{u^\star - u_\coarse^\star} \le \Crel \, \eta_\coarse(u_\coarse^\star)$ for the exact discrete solution.
\bf
\item discrete reliability:\label{axiom:discrete_reliability}\rm\, 
$\enorm{u_\fine^\star - u_\coarse^\star} \le \Crel \, \eta_\coarse(\TT_\coarse \backslash \TT_\fine, u_\coarse^\star)$ for the exact discrete solutions.
\end{enumerate}

\subsection{Discrete iterative solver}
\label{section:solver}

For all $\TT_\coarse \in \T$, let $\Phi_\coarse : \XX_\coarse \to \XX_\coarse$ be the iteration function of  one step of the iterative solver. We require one of the following contraction properties with some uniform constant $0 < \qctr < 1$, which is independent of $\TT_\coarse$:

\renewcommand{\theenumi}{C\arabic{enumi}}
\begin{enumerate}
\bf
\item energy contraction:\label{axiom:energy}\rm\,
$\EE(\Phi_\coarse(v_\coarse)) - \EE(u_\coarse^\star)
\le \qctr^2 \, \big( \EE(v_\coarse) - \EE(u_\coarse^\star) \big)$
for all  $v_\coarse \in \XX_\coarse$.
\bf
\item norm contraction:\label{axiom:norm}\rm\,
$\enorm{u_\coarse^\star - \Phi_\coarse(v_\coarse)}
\le \qctr \, \enorm{u_\coarse^\star - v_\coarse}$
for all  $v_\coarse \in \XX_\coarse$.
\end{enumerate}
To formulate the stopping criterion for the iterative solver, let
\begin{align}\label{eq:def:dist}
 \dist(w,v) := \begin{cases}
 |\EE(v)-\EE(w)|^{1/2} \quad& \text{in case of~\eqref{axiom:energy}}, 
 \\
 \enorm{w-v} \quad& \text{in case of~\eqref{axiom:norm}}.
 \end{cases}
\end{align}
Then, the following lemma provides the means to stop the iterative solver.

\begin{lemma}\label{lemma:lambda}
Let $\TT_\coarse \in \T$ and $v_\coarse \in \XX_\coarse$. Then,~\eqref{axiom:energy} or~\eqref{axiom:norm} imply the following estimates:
\begin{itemize}
\item[\rm(i)] $\dist(u_\coarse^\star, \Phi(v_\coarse)) \le \qctr \, \dist(u_\coarse^\star, v_\coarse)$,
\item[\rm(ii)] $\dist(v_\coarse, \Phi(v_\coarse)) \le (1 + \qctr) \, \dist(u_\coarse^\star, v_\coarse)$,
\item[\rm(iii)] $\dist(u_\coarse^\star, v_\coarse) \le (1 - \qctr)^{-1} \, \dist(v_\coarse, \Phi(v_\coarse))$.
\end{itemize}
\end{lemma}

\begin{proof}
By definition of $\dist(\cdot,\cdot)$, the claim~(i) holds by assumption on~\eqref{axiom:energy} resp.~\eqref{axiom:norm}.
Moreover, (ii)--(iii) follow from the triangle inequality. Note that, if~\eqref{axiom:energy} is valid, 
then $\dist(\cdot,\cdot)$ is a quasi-metric, i.e.,
it holds that $\dist(v_\coarse, v_\coarse) = 0$, $\dist(v_\coarse, w_\coarse) = \dist(w_\coarse, v_\coarse)$, and $\dist(v_\coarse, z_\coarse) \le \dist(v_\coarse, w_\coarse) + \dist(w_\coarse, z_\coarse)$ for all $v_\coarse, w_\coarse, z_\coarse \in \XX_\coarse$. 
\end{proof}

\subsection{Adaptive algorithm}
\label{section:algorithm}

For analytical reasons, we reformulate the adaptive algorithm from the introduction.
More precisely, instead of the single index $n$, we will employ a lower index $\ell$ for the adaptive mesh-refinement as well as an upper index $k$ for the respective steps of the iterative solver.

\begin{algorithm}\label{algorithm}
{\bfseries Input:} Initial mesh $\TT_0$ and $u_0^0 \in \XX_0$, adaptivity parameters $0 < \theta \le 1$, $\lctr > 0$, and $\Cmark \ge 1$, counters $\ell := 0 =: k$.\\
{\bfseries Loop:} Iterate the following Steps~{\rm(i)--(vii)}:
\begin{itemize}
\item[\rm(i)] Update counter $(\ell,k)\mapsto(\ell,k+1)$.
\item[\rm(ii)] Do one step of the iterative solver to obtain $u_\ell^k:=\Phi_\ell(u_\ell^{k-1})$.
\item[\rm(iii)] Compute the local contributions $\eta_\ell(T,u_\ell^k)$ of the error estimator for all $T\in\TT_\ell$.
\item[\rm(iv)] If $\dist(u_\ell^k, u_\ell^{k-1})>\lctr\,\eta_\ell(u_\ell^k)$, continue with {\rm (i)}.
\item[\rm(v)] Otherwise, define $\k(\ell):=k$ and determine a set $\MM_\ell\subseteq\TT_\ell$ with up to the multiplicative factor $\Cmark$ minimal cardinality such that $\theta\,\eta_\ell(u_\ell^k)\leq \eta_\ell(\MM_\ell,u_\ell^k)$.
\item[\rm(vi)] Generate $\TT_{\ell+1}:=\refine(\TT_\ell,\MM_\ell)$ and define $u_{\ell+1}^0:= u_\ell^{\k(\ell)}$.
\item[\rm(vii)] Update counter $(\ell,k)\mapsto(\ell+1,0)$ and continue with {\rm (i)}.
\end{itemize}
{\bfseries Output:} Sequences of successively refined triangulations $\TT_\ell$, discrete solutions $u_\ell^k$, and corresponding error estimators $\eta_\ell(u_\ell^k)$, for all $\ell \geq 0$ and $k\geq 0$.
\end{algorithm}

Define the index set 
$$
\QQ := \set{(\ell,k) \in \N_0^2}{\text{index pair $(\ell,k)$ is used in Algorithm~\ref{algorithm}} \text{ and } k < \k(\ell)}.
$$
Since $u_{\ell + 1}^0 = u_\ell^{\k(\ell)}\!$, we exclude $(\ell,\k(\ell))$ from $\QQ$, if $(\ell + 1, 0) \in \QQ$.
Since Algorithm~\ref{algorithm} is sequential, the index set $\QQ$ is naturally ordered. For $(\ell,k), (\ell',k') \in \QQ$,  we write 
\begin{align}
 (\ell',k') < (\ell,k)
 \quad \stackrel{\text{def}}{\Longleftrightarrow} \quad
 (\ell',k') \text{ appears earlier in Algorithm~\ref{algorithm} than } (\ell,k).
\end{align}
With this order, we can define the \emph{total step counter}
\begin{align*}
 |(\ell,k)| := \# \set{(\ell',k') \in \QQ}{(\ell',k') < (\ell,k)}=k+\sum_{\ell'=0}^{\ell-1} \k(\ell'),
\end{align*}
which provides the total number of solver steps up to the computation of $u_\ell^k$.
Then, provided that~\eqref{axiom:norm} holds, Algorithm~\ref{algorithm} and the adaptive algorithm from the introduction are related through 
$\widetilde\TT_n = \TT_\ell$ and $\widetilde u_n = u_\ell^k$, where $n = |(\ell,k)|$ and quantities with tilde (i.e., $\widetilde\TT_n$ and $\widetilde u_n$) belong to the algorithm from the introduction.

To abbreviate notation, we make the following convention: If the mesh index $\ell \in \N_0$ is clear from the context, we simply write $\k := \k(\ell)$, e.g., $u_\ell^\k := u_\ell^{\k(\ell)}$.
Finally, we introduce some further notation: Define $\underline\ell := \sup\set{\ell \in \N_0}{(\ell,0) \in \QQ}$. Generically, it holds that $\underline\ell = \infty$, i.e., infinitely many steps of mesh-refinement occur. Moreover, for $(\ell,0) \in \QQ$, define $\k(\ell) := \sup\set{k \in \N_0}{(\ell,k) \in \QQ} +1$. We note that the latter definition is consistent with that of Algorithm~\ref{algorithm}, but additionally defines $\k(\underline\ell) = \infty$ if $\underline\ell < \infty$.

\subsection{Abstract main results}
\label{section:main-results}

This section states our main results in the abstract framework of Section~\ref{section:abstract}. We stress that the analysis relies only on the assumptions~\eqref{axiom:sons}--\eqref{axiom:closure} on the mesh-refinement, \eqref{axiom:stability}--\eqref{axiom:discrete_reliability} on the error estimator, and~\eqref{axiom:energy} resp.\ \eqref{axiom:norm} on the iterative solver. We refer to Section~\ref{section:linear} and Section~\ref{section:nonlinear} below, where these assumptions are verified for concrete model problems.

First, we note that due to the contraction property~\eqref{axiom:energy} resp.\ \eqref{axiom:norm}, we have {\sl a~posteriori} error control of the error.
The proof is given in Section~\ref{section:reliable}.

\begin{proposition}\label{lemma:reliable}
Suppose~\eqref{axiom:energy} or~\eqref{axiom:norm}. Suppose~\eqref{axiom:stability}--\eqref{axiom:reliability}. Then, the quasi-error
\begin{align}\label{eq:quasi-error}
 \Delta_\ell^k:=\enorm{u^\star-u_\ell^k}+\eta_\ell(u_\ell^k)
 \quad\text{for all } (\ell,k) \in \overline\QQ := \QQ \cup \set{(\ell,\k)}{\k(\ell) < \infty}
\end{align}
satisfies that
\begin{align}\label{eq:lemma:reliable}
 \Delta_\ell^k 
 \le \Crel' \begin{cases}
  \eta_\ell(u_\ell^k) + \dist(u_\ell^k, u_\ell^{k-1}) \quad & \text{if \ }0 < k \le \k(\ell), \\
  \eta_\ell(u_\ell^\k) \quad & \text{if \ } k = \k(\ell), \\
  \eta_{\ell-1}(u_\ell^{0}) \quad & \text{if \ } k = 0 \text{ and }\ell > 0.
 \end{cases}
\end{align}
The constant $\Crel' > 0$ depends only on $\Cstab$, $\Crel$, $\qctr$, and $\lctr$ under~\eqref{axiom:norm}, while it additionally depends on $\alpha$ under~\eqref{axiom:energy}.
\end{proposition}

The first main theorem states linear convergence of the quasi-error. We note that under certain assumptions, linear convergence holds for arbitrary parameters $0 < \theta \le 1$ and $\lctr > 0$.
The proof is given in Section~\ref{section:proofs}.

\begin{theorem}\label{theorem:linconv}
Suppose~\eqref{axiom:energy} or~\eqref{axiom:norm}. Suppose~\eqref{axiom:stability}--\eqref{axiom:reliability}.
Define
\begin{align}\label{def:lambda}
 \lconv := \begin{cases} 
 \infty &\text{if~\eqref{axiom:energy} is valid},\\
 \frac{1-\qctr}{\Cstab \qctr} \quad &\text{otherwise}.
 \end{cases}
\end{align}
Then, for all $0 < \theta \leq 1$ and $0 < \lctr < \lconv \theta$, there exist constants $\Clin \geq 1$ and $0 < \qlin < 1$ such that the quasi-error~\eqref{eq:quasi-error}
is linearly convergent in the sense of
\begin{align}\label{eq:theorem:linconv}
 \Delta_{\ell}^{k} \leq \Clin \, \qlin^{|(\ell,k)|-|(\ell',k')|}\,\Delta_{\ell'}^{k'}
 \qquad\textrm{ for all }
 (\ell, k), (\ell', k') \in \QQ \textrm{ with } (\ell', k') < (\ell, k).
\end{align}
The constants $\Clin$ and $\qlin$ depend only on $\Ccea=L/\alpha$, $\Cstab$, $\qred$, $\Crel$, $\qctr$, and the adaptivity parameters $\theta$ and $\lctr$, while it additionally depends on $L$ in case of \eqref{axiom:energy}.
\end{theorem}

\begin{remark}
Theorem~\ref{theorem:linconv} is remarkably stronger than the corresponding results in~\cite{banach}:

{\rm(i)} \cite[Theorem~5.3]{banach} assumes that $\k(\ell) < \infty$ for all $\ell \in \N_0$ and then proves linear convergence only for the final iterates, i.e., $\eta_{\ell+m}(u_{\ell+m}^\k) \lesssim \qlin^m \, \eta_\ell(u_\ell^\k)$ for all $\ell, m \in \N_0$. If $\k(\ell) = \infty$ for some $\ell \in \N_0$, then~\cite[Proposition~4.4]{banach} proves only plain convergence $\Delta_\ell^k \to 0$ as $k \to \infty$. Instead, the present Theorem~\ref{theorem:linconv} states linear convergence of $\Delta_\ell^k$ in any case as $|(\ell,k)| \to \infty$.

{\rm(ii)} Moreover, \cite[Theorem~5.3]{banach} is always constrained by $\lconv = \frac{1-\qctr}{\Cstab \qctr}$, while the present Theorem~\ref{theorem:linconv} allows even for $\lconv = \infty$ in certain (relevant) situations.
\end{remark}

The following corollary states that the exact solution $u^\star$ is discrete if $\underline\ell<\infty$, i.e., if the number of mesh refinements is bounded.

\begin{corollary}\label{cor:linconv}
Suppose the assumptions of Theorem~\ref{theorem:linconv}.
Then, $\underline\ell < \infty$ implies that $u^\star = u_{\underline\ell}^\star$ and $\eta_{\underline\ell}(u_{\underline\ell}^\star) = 0$.
\end{corollary}

\begin{proof}
According to Theorem~\ref{theorem:linconv}, it holds that 
$$ 
 \enorm{u^\star - u_{\underline\ell}^k} + \eta_{\underline\ell}(u_{\underline\ell}^k) = \Delta_{\underline\ell}^k \to 0
 \quad \text{as } k \to \infty.
$$
Moreover, contraction~\eqref{axiom:energy} or \eqref{axiom:norm}  (together with  \eqref{eq:energy} in case of  \eqref{axiom:energy}) prove that 
$$
 \enorm{u_{\underline\ell}^\star - u_{\underline\ell}^k} 
 \simeq \dist(u_{\underline\ell}^\star, u_{\underline\ell}^k)
 \le \qctr^k \, \dist(u_{\underline\ell}^\star, u_{\underline\ell}^0)
 \to 0
 \quad \text{as } k \to \infty.
$$
Uniqueness of the limit yields that $u_{\underline\ell}^\star = u^\star$. Moreover, it follows that
$$
 0 \le \eta_{\underline\ell}(u_{\underline\ell}^\star) 
 \reff{axiom:stability} \le \eta_{\underline\ell}(u_{\underline\ell}^k) + \enorm{u_{\underline\ell}^\star - u_{\underline\ell}^k} \to 0 
  \quad \text{as } k \to \infty.
$$
This concludes the proof.
\end{proof}

The second main theorem states optimal convergence rates of the quasi-error \eqref{eq:quasi-error} with respect to the overall computational costs. As usual in this context (see, e.g.,~\cite{axioms}), the result requires that the adaptivity parameters $0 < \theta \le 1$ and $\lctr > 0$ are sufficiently small. 
The proof is given in Section~\ref{section:costs}.

\begin{theorem}\label{theorem:rates}
Suppose~\eqref{axiom:energy} or~\eqref{axiom:norm}. Suppose~\eqref{axiom:sons}--\eqref{axiom:closure} and~\eqref{axiom:stability}--\eqref{axiom:discrete_reliability}. Define
\begin{align}\label{def:lambda_opt}
 \lopt := \begin{cases} 
 \frac{1 - \qctr}{\qctr \Cstab} &\text{if~\eqref{axiom:norm} is valid},\\
 \frac{1 - \qctr}{\qctr \Cstab}\,\sqrt{\alpha/2} \quad &\text{otherwise}.
 \end{cases}
\end{align}
Let $0 < \theta \leq 1$ and $0 < \lctr < \lopt \theta$ such that
\begin{align}\label{eq:opt:theta'}
 0 < \theta' := \frac{\theta + \lctr/\lopt}{1 - \lctr/\lopt} < (1 + \Cstab^2 \Crel^2)^{-1/2}.
\end{align}
Let $s > 0$. 
Then, there exist $\copt, \Copt > 0$ such that
\begin{align}\label{eq:optimal}
\begin{split}
 \copt^{-1} \, \norm{u^\star}{\mathbb{A}_s} 
 &\le  \sup_{(\ell',k') \in \QQ} (\#\TT_{\ell'} - \#\TT_0 + 1)^s \, \Delta_{\ell'}^{k'}
 \\ &
 \le \sup_{(\ell',k') \in \QQ} \bigg(\sum_{\substack{(\ell,k) \in \QQ \\ (\ell,k) \le (\ell',k')}} \#\TT_\ell \bigg)^s \, \Delta_{\ell'}^{k'}
 \le \Copt \, \max\{\norm{u^\star}{\mathbb{A}_s},\Delta_0^0\},
\end{split}
\end{align}
where $\norm{u^\star}{\mathbb{A}_s}$ is defined in~\eqref{eq:As}.
The constant $\copt > 0$ depends only on $\Ccea=L/\alpha$, $\Cson$, $\Cstab$, $\Crel$, $\#\TT_0$, and $s$, and additionally on $\underline\ell$ resp.\ $\ell_0$, if $\underline\ell<\infty$ or $\eta_{\ell_0}(u_{\ell_0}^{\underline k})=0$ for some $(\ell_0+1,0)\in\QQ$.
The constant $\Copt > 0$ depends only on $\Cstab$, $\qred$, $\Crel$, $\Cmesh$, $1 - \lctr/\lopt$, $\Cmark$, $\Crel'$, $\Clin$, $\qlin$, $\#\TT_0$, and~$s$.
\end{theorem}

\bigskip

\subsection{Remarks on abstract assumptions}
\label{section:remarks}
In this section, we briefly comment on the validity of the abstract assumptions made. 

{\it\textbf{Mesh-refinement~\bf(\ref{axiom:sons})--(\ref{axiom:closure}).}}
\label{rem:refinement axioms}
 For mesh-refinement of simplical meshes by newest vertex bisection, the properties~\eqref{axiom:sons}--\eqref{axiom:closure} are verified in~\cite{bdd2004,stevenson2007,stevenson2008,ckns2008,gss2014}. We note that the mesh-closure estimate~\eqref{axiom:closure} requires a technical \emph{admissibility condition} on $\TT_0$ for $d \ge 3$, which is proven unnecessary for $d = 2$ in~\cite{kpp2013}. We refer to~\cite{bn2010} for red-refinement with first-order hanging nodes.
 In the frame of isogeometric analysis, we mention the mesh-refinement techniques for  analysis-suitable T-splines \cite{mp15}, truncated hierarchical B-splines \cite{bgmp16}, and hierarchical B-splines \cite{ghp17}. 

\textbf{\textit{Error estimator}~(\ref{axiom:stability})--(\ref{axiom:discrete_reliability}).}
 The verification of~\eqref{axiom:stability}--\eqref{axiom:discrete_reliability} in Section~\ref{section:linear} and \ref{section:nonlinear} relies on scaling arguments and implicitly uses that all meshes $\TT_\coarse \in \T$ are uniformly shape regular. Moreover, we note that the analysis is implicitly tailored to weighted-residual error estimators, since the usual verification of~\eqref{axiom:reduction} relies on exploiting the contraction of the mesh-size on refined elements.

\textbf{\textit{Iterative solver}~(\ref{axiom:energy})--(\ref{axiom:norm}).}
\label{rem:solver axioms}
For linear symmetric problems, one usually has that $\EE(v_\coarse) - \EE(u_\coarse^\star) = \frac{1}{2} \, \enorm{v_\coarse - u_\coarse^\star}^2$, and hence~\eqref{axiom:energy} and~\eqref{axiom:norm} are equivalent.

In the setting of strongly monotone and Lipschitz continuous nonlinear operators (see Section~\ref{section:abstract}), the Zarantonello (or Banach--Picard) iteration $\Phi_\coarse : \XX_\coarse \to \XX_\coarse$ defined by 
\begin{align}\label{eq:definition:phi}
 \product{\Phi_\coarse(v_\coarse)}{w_\coarse}
 = \product{v_\coarse}{w_\coarse}
 - \frac{\alpha}{L^2} \dual{\AA v_\coarse - F}{w_\coarse}_{\HH'\times\HH}
 \quad \text{for all } w_\coarse \in \XX_\coarse
\end{align}
satisfies~\eqref{axiom:norm} with $\qctr^2 = 1-\alpha^2/L^2$; see, e.g.,~\cite{cw2017,banach,hw18,hw19}. Hence, 
\begin{align*}
 \EE(\Phi_\coarse(v_\coarse)) \!-\! \EE(u_\coarse^\star)
 \reff{eq:energy}\le
 \frac{L}{2} \, \enorm{u_\coarse^\star \!-\! \Phi_\coarse(v_\coarse)}^2
 \reff{axiom:norm}\le \frac{L}{2} \, \qctr^2 \, \enorm{u_\coarse^\star \!-\! v_\coarse}^2
 \reff{eq:energy}\le \frac{L}{\alpha} \, \qctr^2 \, \big( \EE(v_\coarse) \!-\! \EE(u_\coarse^\star) \big).
\end{align*}
In this case, the additional validity of~\eqref{axiom:energy} with the modified constant $\frac{L}{\alpha} \, \qctr^2$ follows from an additional condition on $L/\alpha$ involving the \emph{golden ratio}, namely
\begin{align}\label{eq:golden_ratio}
 0 \le \frac{L}{\alpha} \, \qctr^2 = \frac{L}{\alpha} - \frac{\alpha}{L} < 1 
 \quad \Longleftrightarrow \quad
 \frac{L}{\alpha} < \frac{1+ \sqrt{5}}{2} \approx 1.618.
\end{align}
Moreover, with the same arguments,~\eqref{axiom:energy} guarantees that
\begin{align*}
 \enorm{u_\coarse^\star \!-\! \Phi_\coarse(v_\coarse)}^2
 \le \frac{L}{\alpha} \, \qctr^2 \, \enorm{u_\coarse^\star  \!-\! v_\coarse}^2.
\end{align*}
Hence, the condition~\eqref{eq:golden_ratio} even yields equivalence of~\eqref{axiom:energy} and~\eqref{axiom:norm} (but with different contraction constants $\qctr$).

\subsection{AFEM for linear elliptic PDE with optimal PCG solver}
\label{section:linear}
We consider the boundary value problem \eqref{eq:strong}, where we assume that 
\begin{align}
A:L^2(\Omega)^d\to L^2(\Omega)^d \quad\text{has the form} \quad A({\bf v})=\big[ x\mapsto{\bf A}(x){\bf v}(x) \big],
\end{align}
 where ${\bf A}\in W^{1,\infty}(\Omega)^{d\times d}$ is symmetric and uniformly positive definite. 
Note that ${\bf A}\in W^{1,\infty}(\Omega)^{d\times d}$ instead of ${\bf A}\in L^{\infty}(\Omega)^{d\times d}$ is only necessary to ensure that the residual error indicators~\eqref{eq:indicators for linear afem} are well-defined.
One easily checks that 
\begin{align}\label{eq:potential for linear afem}
P:H_0^1(\Omega)\to\R, \quad v\mapsto \frac{1}{2}\int_\Omega{\bf A}\nabla v \cdot\nabla v\,{\rm d}x
\end{align}
 satisfies \eqref{eq:potential}. 
If we equip $H_0^1(\Omega)$ with the scalar product 
\begin{align}\label{eq:product for linear afem}
\product{w}{v}:=\int_\Omega{\bf A}\nabla w \cdot\nabla v\,{\rm d}x,
\end{align}
then \eqref{eq:monotone+lipschitz} is satisfied with $\alpha = 1 = L$. 

Let  $\TT_0$ be a conforming initial triangulation of $\Omega$ into simplices $T\in\TT_0$.
We employ newest vertex bisection for $\refine$.
According to Remark~\ref{rem:refinement axioms}, \eqref{axiom:sons}--\eqref{axiom:closure} are satisfied. 
For each $\TT_\coarse\in\T$, we define the corresponding space 
\begin{align}
\XX_\coarse:=\set{v\in C(\Omega)}{v|_\Gamma=0 \text{ and }v|_T\in\mathcal{P}^p\text{ for all }T\in\TT_\coarse} 
\end{align}
as the space of all continuous piecewise polynomials of  fixed degree $p\ge1$ that vanish on the boundary $\Gamma=\partial\Omega$.
We define the weighted-residual error indicators (see, e.g., \cite{ao11,verfuerth13}) for all $T\in\TT_\coarse$ and $v_\coarse\in\XX_\coarse$ as
\begin{align}\label{eq:indicators for linear afem}
\eta_\coarse(T,v_\coarse)^2:= |T|^{2/d}\norm{f+\div ({\bf A}\nabla v_\coarse)}{L^2(T)} 
+|T|^{1/d}\norm{[{\bf A}\nabla v_\coarse\cdot {\bf n}]}{L^2(\partial T\cap \Omega)},
\end{align}
where $[\cdot]$ denotes the usual jump of piecewise continuous functions across element interfaces, and ${\bf n}$ is the outer normal vector of the considered element. 
It is well-known that the resulting error estimator satisfies the axioms \eqref{axiom:stability}--\eqref{axiom:discrete_reliability}; see, e.g., \cite[Section~6.1]{axioms} and the references therein. 

Finally, we introduce the iteration function $\Phi_\coarse:\XX_\coarse\to\XX_\coarse$ as one step of PCG: 
Let
\begin{align} 
{\bf M}_\coarse:=\Big(\int_\Omega{\bf A}\nabla \zeta_j \cdot\nabla\zeta_i\,{\rm d}x\Big)_{i,j=1}^N
\in\R^{N\times N}
\end{align}
 be the Galerkin matrix corresponding to the usual Lagrangian basis $\{\zeta_1,\dots,\zeta_N\}$ of $\XX_\coarse$ and ${\bf P}_\coarse\in\R^{N\times N}$ be an arbitrary symmetric positive definite  preconditioner. 
With the right-hand side vector 
\begin{align}
{\bf b}_\coarse:=\Big( \int_\Omega f \,\zeta_i \,{\rm d}x\Big)_{i=1}^N \in \R^N
\end{align}
 corresponding to \eqref{eq:discrete}, the coefficient vector ${\bf x}_\coarse^\star\in\R^N$ of $u_\coarse^\star=\sum_{i=1}^N {\bf x}_\coarse^\star[i]\,\zeta_i$ is the unique solution of the linear system
\begin{align}\label{eq:linearsystem}
{\bf M}_\coarse {\bf x}_\coarse^\star={\bf b}_\coarse.
\end{align}
For any $v_\coarse\in\XX_\coarse$ with coefficient vector ${\bf y}_\coarse\in\R^N$, we have the elementary identity 
\begin{align}\label{eq:normequi2}
 \enorm{v_\coarse}^2 
 = {\bf y}_\coarse \cdot {\bf M}_\coarse {\bf y}_\coarse
 =:  |{\bf y}_\coarse|_ {{\bf M}_\coarse}^2.
\end{align}
Given an initial guess ${\bf x}_{\coarse}^0$, PCG (see~\cite[Algorithm~11.5.1]{matcomp}) approximates the solution ${\bf x}_\coarse^\star \in \R^N$.
We note that each step of PCG has the following computational costs:
\begin{itemize}
\item $\OO(N)$ costs for vector operations (e.g., assignment, addition, scalar product),
\item computation of \emph{one} matrix-vector product with ${\bf M}_\coarse$,
\item computation of \emph{one} matrix-vector product with ${\bf P}_\coarse^{-1}$.
\end{itemize}
PCG formally applies the conjugate gradient method (CG, see~\cite[Algorithm~11.3.2]{matcomp}) for the matrix $\widetilde{\bf M}_\coarse := {\bf P}_\coarse^{-1/2} {\bf M}_\coarse {\bf P}_\coarse^{-1/2}$ and the right-hand side $\widetilde{\bf b}_\coarse := {\bf P}_\coarse^{-1/2} {\bf b}_\coarse$. The iterates ${\bf x}_{\coarse}^k \in \R^N$ of PCG (for ${\bf P}_\coarse$, ${\bf M}_\coarse$, ${\bf b}_\coarse$, and the initial guess ${\bf x}_{\coarse}^0$) and the iterates $\widetilde{\bf x}_{\coarse}^k$ of CG (for $\widetilde{\bf M}_\coarse$, $\widetilde{\bf b}_\coarse$, and the initial guess $\widetilde{\bf x}_{\coarse}^0 := {\bf P}_\coarse^{1/2} {\bf x}_{\coarse}^0$) are formally linked by
\begin{align*}
 {\bf x}_{\coarse}^k = {\bf P}_\coarse^{-1/2} \widetilde{\bf x}_{\coarse}^k;
\end{align*}
see~\cite[Section~11.5]{matcomp}. Moreover, direct computation proves that
\begin{align}\label{eq:normequi1}
|\widetilde{\bf y}_\coarse|_{\widetilde{\bf M}_\coarse}^2 := \widetilde{\bf y}_\coarse \cdot \widetilde{\bf M}_\coarse \widetilde{\bf y}_\coarse 
 =  |{\bf y}_\coarse|_{{\bf M}_\coarse}^2
 \quad \text{for all $\widetilde {\bf y}_\coarse \in \R^N$ and ${\bf y}_\coarse = {\bf P}_\coarse^{-1/2} \widetilde {\bf y}_\coarse$}.
\end{align}
Hence,~\cite[Theorem~11.3.3]{matcomp} for CG (applied to $\widetilde{\bf M}_\coarse$, $\widetilde{\bf b}_\coarse$,  $\widetilde{\bf x}_{\coarse0}$) yields the following lemma for PCG (which follows from the implicit steepest decent property of CG).  

\begin{lemma}\label{lemma:pcg}
Let ${\bf M}_\coarse, {\bf P}_\coarse \in \R^{N \times N}$ be symmetric and positive definite, ${\bf b}_\coarse\in \R^N$, ${\bf x}_\coarse^\star := {\bf M}_\coarse^{-1} {\bf b}_\coarse$, and ${\bf x}_{\coarse}^0 \in \R^N$. 
Suppose the $\ell_2$-condition number estimate
\begin{align}\label{eq1:pcg}
 {\rm cond}_2({\bf P}_\coarse^{-1/2} {\bf M}_\coarse {\bf P}_\coarse^{-1/2}) \le \Cpcg.
\end{align}
Then, the iterates ${\bf x}_{\coarse}^k$ of the PCG algorithm satisfy the contraction property
\begin{align}\label{eq2:pcg}
|{\bf x}_\coarse^\star - {\bf x}_{\coarse}^{k+1}|_{{\bf M}_\coarse} 
 \le \qpcg \, |{\bf x}_\coarse^\star - {\bf x}_{\coarse}^k|_{{\bf M}_\coarse}
 \quad\text{for all } k \in \N_0,
\end{align}
where $\qpcg := (1-1/\Cpcg)^{1/2} < 1$.\qed
\end{lemma}

Finally, we suppose that the employed preconditioners ${\bf P}_\coarse$ are optimal in the sense that $\Cpcg$ depends only onthe coefficient matrix ${\bf A}$, the initial mesh $\TT_0$, and  the polynomial degree $p$. 
We stress that such optimal symmetric positive preconditioners exist and the product of  ${\bf P}_\coarse$ with \emph{one} vector can be realized in linear complexity $\OO(N)$; see, e.g., \cite{wc06,smpz08,xch10,cnx12}.
Then, \eqref{eq2:pcg} together with  \eqref{eq:normequi2} immediately gives the norm contraction \eqref{axiom:norm}. 
Due to \eqref{eq:potential for linear afem}--\eqref{eq:product for linear afem}, we have that $|\mathcal{E}(v)-\mathcal{E}(w)|=(1/2)\enorm{w-v}^2$ for all $v,w\in H_0^1(\Omega)$, and thus \eqref{axiom:norm} is equivalent to \eqref{axiom:energy}.
Altogether, the main results from Section~\ref{section:main-results} apply to the present setting. Moreover, the linear convergence \eqref{eq:theorem:linconv} from Theorem~\ref{theorem:linconv} holds even for arbitrary $\lctr>0$ and $0<\theta\le1$ in Algorithm~\ref{algorithm}.

\subsection{AFEM for strongly monotone nonlinearity}
\label{section:nonlinear}
We consider the boundary value problem \eqref{eq:strong}, where we assume that 
\begin{align}
A:L^2(\Omega)^d\to L^2(\Omega)^d \quad \text{has the form}\quad A({\bf v})=\big[ x\mapsto{a}(x,|{\bf v}(x)|^2) \, {\bf v}(x) \big]
\end{align}
 with a scalar nonlinearity ${a}: \Omega \times \R_{\geq 0} \rightarrow \R$ that satisfies the following properties~\eqref{item:G_bounded}--\eqref{item:G_lipschitz1} with generic constants $c_a, C_a, c_a', C_a', L_a, L_a'>0$, which have already been considered in~\cite{gmz2012,banach}:

 \renewcommand{\theenumi}{N\arabic{enumi}}
\begin{enumerate}
 \bf
 \item \label{item:G_bounded}
 \rm
{\bf boundedness of $\boldsymbol{a(x,t)}$:} $ c_a \le {a}(x,t) \le C_a$  for all $x \in \Omega$ and all  $t \geq 0$.
 \bf
 \item \label{item:G_differentiable}
 \rm
{\bf boundedness of $\boldsymbol{{a}(x,t) +2 t  \frac{\d{}}{\d{t}} {a}(x,t)}$:}  ${a}(x,\cdot) \in C^1(\R_{\geq 0} ,\R)$ for all $x \in \Omega$, and  
$ c_a' \le {a}(x,t) +2 t  \frac{\d{}}{\d{t}} {a}(x,t) \le C_a'$  for all  $x \in \Omega$  and all  $t \geq 0$.
  \bf
 \item \label{item:G_lipschitz}
 \rm
{\bf Lipschitz-continuity of $\boldsymbol{a(x,t)}$ in $\boldsymbol{x}$:}
$ | {a}(x,t) - {a}(y,t) | \le L_{{a}} | x - y |$  for all  $x,y \in \Omega$  and all  $t \geq 0$.
 \bf
 \item \label{item:G_lipschitz1}
 \rm
{\bf Lipschitz-continuity of $\boldsymbol{t \frac{\d{}}{\d{t}} {a}(x,t)}$ in $\boldsymbol{x}$:} 
$| t \frac{\d{}}{\d{t}} {a}(x,t) - t \frac{\d{}}{\d{t}} {a}(y,t) | \le {L}_{{a}}' | x - y |$  for all  $x,y \in \Omega$
  and all  $t \geq 0$.
\end{enumerate}

According to, e.g., \cite[Proposition~8.2]{banach}, \eqref{item:G_bounded}--\eqref{item:G_differentiable} imply the existence of some $P: H_0^1(\Omega)\to\R$ with \eqref{eq:potential}--\eqref{eq:monotone+lipschitz}, where $\alpha=c_a'$ and $L=C_a'$ if $H_0^1(\Omega)$ is equipped with the scalar product $\product{w}{v}:=\int_\Omega\nabla w\cdot\nabla v\,{\rm d}x$. 

Let  $\TT_0$ be a conforming initial triangulation of $\Omega$ into simplices $T\in\TT_0$.
We employ newest vertex bisection for $\refine$.
According to Remark~\ref{rem:refinement axioms}, \eqref{axiom:sons}--\eqref{axiom:closure} are satisfied. 
For each $\TT_\coarse\in\T$, consider the lowest-order FEM space 
\begin{align}\label{eq:nonlinear:Xh}
\XX_\coarse:=\set{v\in C(\Omega)}{v|_\Gamma=0 \text{ and }v|_T\in\mathcal{P}^1\text{ for all }T\in\TT_\coarse}.
\end{align} 
We define the weighted-residual error indicators (see, e.g., \cite{gmz2012,banach}) for all $T\in\TT_\coarse$ and $v_\coarse\in\XX_\coarse$ as
\begin{align}
\begin{split}
\eta_\coarse(T,v_\coarse)^2:= &|T|^{2/d}\norm{f+\div ({a}(\cdot,|\nabla v_\coarse|^2) \nabla v_\coarse)}{L^2(T)} \\
+&|T|^{1/d}\norm{[{a}(\cdot,|\nabla v_\coarse|^2)\nabla v_\coarse)\cdot {\bf n}]}{L^2(\partial T\cap \Omega)},
\end{split}
\end{align}
where $[\cdot]$ denotes the usual jump of piecewise continuous functions across element interfaces, and ${\bf n}$ is the outer normal vector of the considered element. 
Note that \eqref{item:G_lipschitz} guarantees that the presented error indicators are well-defined. 
Reliability~\eqref{axiom:reliability} and discrete reliability~\eqref{axiom:discrete_reliability} are proved as in the linear case; see, e.g.,~\cite{ckns2008} for the linear case and~\cite[Theorem~3.3 and 3.4]{gmz2012} for 
the present nonlinear setting.

The verification of stability~\eqref{axiom:stability} and reduction~\eqref{axiom:reduction} requires the validity of an appropriate inverse estimate. For scalar nonlinearities and under the assumptions~\eqref{item:G_bounded}--\eqref{item:G_lipschitz1}, the latter is proved in~{\cite[Lemma 3.7]{gmz2012}}. Using this inverse estimate, the proof of~\eqref{axiom:stability}--\eqref{axiom:reduction}
follows as for the linear case; see, e.g.,~\cite{ckns2008} for the linear case or \cite[Section~3.3]{gmz2012} for scalar nonlinearities. We note that the necessary inverse estimate is still open for non-scalar nonlinearities and/or higher polynomial order $p \ge 2$.

As iteration function $\Phi_\coarse:\XX_\coarse\to\XX_\coarse$, we employ the Zarantonello iteration from~\eqref{eq:definition:phi}.
In the present setting,  $\Phi_\coarse(v_\coarse)\in\XX_\coarse$ is obtained by solving the linear system
\begin{align}
\begin{split}
\int_\Omega \nabla \Phi_\coarse(v_\coarse)\cdot \nabla w_\coarse\,\mathrm{d}x
 = \int_\Omega \big(1- \frac{\alpha}{L^2}{a}(\cdot,|\nabla v_\coarse|^2)\big) \nabla v_\coarse \cdot \nabla w_\coarse \,\mathrm{d}x
+ \frac{\alpha}{L^2} \int_\Omega f w_\coarse \, \mathrm{d}x
 \end{split}
\end{align}
for all $w_\coarse \in \XX_\coarse$.
In explicit terms, the computation of one step of the iteration requires only the solution of one (discretized) Poisson equation with homogeneous Dirichlet data. 

Altogether, the main results from Section~\ref{section:main-results} apply to the present setting.

\section{Proof of Proposition~\ref{lemma:reliable} (a~posteriori error control)}
\label{section:reliable}

For $(\ell,k) \in \overline\QQ$ and $k > 0$, it holds that
\begin{align*}
 &\enorm{u^\star - u_\ell^k} 
 \le \enorm{u^\star - u_\ell^\star}+ \enorm{u_\ell^\star - u_\ell^k} 
 \reff{axiom:reliability}\le \Crel \, \eta_\ell(u_\ell^\star) + \enorm{u_\ell^\star - u_\ell^k} 
 \\& \quad
 \reff{axiom:stability}\le \Crel \, \eta_\ell(u_\ell^k) + (\Crel\Cstab + 1) \, \enorm{u_\ell^\star - u_\ell^k}.
\end{align*}
Suppose~\eqref{axiom:norm}.
With Lemma~\ref{lemma:lambda} (i)$\&$(iii), it then follows that
\begin{align*}
 \enorm{u_\ell^\star - u_\ell^k}
 \le \frac{\qctr}{1-\qctr} \, \enorm{u_\ell^k - u_\ell^{k-1}}
 = \frac{\qctr}{1-\qctr} \, \dist(u_\ell^k, u_\ell^{k-1}).
\end{align*}
Suppose~\eqref{axiom:energy}. With~\eqref{eq:energy} and Lemma~\ref{lemma:lambda}(i)$\&$(iii), it then follows that
\begin{align*}
 \enorm{u_\ell^\star - u_\ell^k}
 \reff{eq:energy}\le \sqrt{2/\alpha} \, \dist(u_\ell^\star, u_\ell^k)
 \le \sqrt{2/\alpha} \, \frac{\qctr}{1-\qctr} \, \dist(u_\ell^k, u_\ell^{k-1}).
\end{align*}
Since $\Delta_\ell^k = \enorm{u^\star - u_\ell^k} + \eta_\ell(u_\ell^k)$, this proves~\eqref{eq:lemma:reliable} for the case that $0 < k \le \k(\ell)$. If $k = \k(\ell)$, then the stopping criterion in Algorithm~\ref{algorithm}(iv) yields that
\begin{align*}
 \dist(u_\ell^\k, u_\ell^{\k-1}) \le \lctr \eta_\ell(u_\ell^\k).
\end{align*}
This proves~\eqref{eq:lemma:reliable} for $k = \k(\ell)$. If $k = 0$ and $\ell > 0$, then $u_\ell^0 = u_{\ell-1}^\k$ and hence 
$$
 \enorm{u^\star - u_\ell^0} = \enorm{u^\star - u_{\ell-1}^\k} \lesssim \eta_{\ell-1}(u_{\ell-1}^\k) = \eta_{\ell-1}(u_\ell^0)
$$
follows from the previous step.
Moreover, \eqref{axiom:stability}--\eqref{axiom:reduction} yield that $\eta_\ell(u_\ell^0) \le \eta_{\ell-1}(u_\ell^0)$. 
Since $\Delta_\ell^k = \enorm{u^\star - u_\ell^k} + \eta_\ell(u_\ell^k)$, this concludes the proof.\qed

\section{Proof of Theorem~\ref{theorem:linconv} (linear convergence)}
\label{section:proofs}

\noindent
Recall the definition of $\dist(\cdot,\cdot)$ from~\eqref{eq:def:dist}.
According to Algorithm~\ref{algorithm}, the contractive solver stops for the minimal $k = \k(\ell) \ge 1$ such that 
\begin{align}\label{eq:stopping:linearization}
 \dist(u_\ell^\k, u_\ell^{\k-1})
 \le \lctr \, \eta_\ell(u_\ell^\k).
\end{align}
In particular, this implies that
\begin{align}\label{eq2:stopping:linearization}
 \eta_\ell(u_\ell^k) 
 < \lctr^{-1} \, \dist(u_\ell^k, u_\ell^{k-1})
 \quad \text{for all } (\ell,k) \in \QQ \text{ with } k > 0.
\end{align}

\subsection{Proof of Theorem~\ref{theorem:linconv} under assumption~(\ref{axiom:energy})}
In this section, we give a proof of Theorem~\ref{theorem:linconv} under the assumption that the iterative solver $\Phi_\ell$ leads to a uniform contraction of the discrete energy~\eqref{axiom:energy}. 
Note that 
\begin{align}\label{eq:galerkin}
 \dist(u^\star, v_\coarse)^{2} = \dist(u^\star, u_\coarse^\star)^{2} + \dist(u_\coarse^\star, v_\coarse)^{2}
 \quad \text{for all } v_\coarse \in \XX_\coarse,
\end{align}
since all three energy differences in the \emph{latter equality} are non-negative (see~\eqref{eq:energy}), and hence the absolute values in the definition of $\dist(\cdot, \cdot)$ can be omitted.

\begin{lemma}\label{prop:lambda}
Suppose \eqref{axiom:stability}--\eqref{axiom:reliability} and \eqref{axiom:energy}.
Let $0<\theta\le 1$ and $\lctr>0$.
Then, there exist constants $\mu > 0$ and $0 < \qlin < 1$ such that 
\begin{align}\label{eq:lambda}
\Lambda_\ell^k := \dist(u^\star, u_\ell^k)^{2} + \mu \, \eta_\ell(u_\ell^k)^2
\quad \text{for all } (\ell, k) \in \QQ
\end{align}
satisfies the following statements~{\rm(i)}--{\rm(ii)}:
\begin{itemize}
\item[\rm(i)] $\Lambda_\ell^{k+1} \le \qlin^2 \, \Lambda_\ell^k$ for all $(\ell,k+1) \in \QQ$.
\item[\rm(ii)] $\Lambda_{\ell+1}^0 \le \qlin^2 \, \Lambda_\ell^{\k-1}$ for all $(\ell+1,0) \in \QQ$.
\end{itemize}
The constants $\mu$ and $\qlin$ depend only on $L$, $\alpha$, $\Cstab$, $\qred$, $\Crel$, and $\qctr$ as well as on the adaptivity parameters $0 < \theta \le 1$ and $\lctr > 0$.
\end{lemma}

\begin{proof}[Proof of Lemma~\ref{prop:lambda}\rm(i)]
Let $\mu, \eps > 0$ be free parameters, which will be fixed below.
First, we note that
\begin{align*}
 \enorm{u^\star - u_\ell^\star}^2
 \reff{axiom:reliability}\le \Crel^2 \, \eta_\ell(u_\ell^\star)^2
 \reff{axiom:stability}\le 2 \, \Crel^2 \, \eta_\ell(u_\ell^{k+1})^2
 + 2 \, \Crel^2 \, \Cstab^2 \, \enorm{u_\ell^\star - u_\ell^{k+1}}^2.
\end{align*}
Together with~\eqref{eq:energy}, this leads to
\begin{align*}
 \nonumber
 \dist(u^\star, u_\ell^\star)^{2}
 \reff{eq:energy}\le \frac{L}{2} \, 
 \enorm{u^\star - u_\ell^\star}^2
 &\reff{eq:energy}\le L \Crel^2 \,\eta_\ell(u_\ell^{k+1})^2 
 + 2 L\alpha^{-1} \, \Crel^2 \Cstab^2 \, \dist(u_\ell^\star, u_\ell^{k+1})^{2}.
\end{align*}
Let $C_1 := L \Crel^2$ and $C_2 := 2 \, L\alpha^{-1} \, \Crel^2 \Cstab^2$. With this, we obtain that
\begin{align*}
 \dist(u^\star, u_\ell^{k+1})^{2}
 &\reff{eq:galerkin}= 
  (1-\eps) \, \dist(u^\star, u_\ell^\star)^{2} + \eps \, \dist(u^\star, u_\ell^\star)^{2}
  + \dist(u_\ell^\star, u_\ell^{k+1})^{2}
 \\& 
 \,\le\, 
 (1-\eps) \, \dist(u^\star, u_\ell^\star)^{2}
  + \eps \, C_1 \, \eta_\ell(u_\ell^{k+1})^2 + (1 + \eps \, C_2) \, \dist(u_\ell^\star, u_\ell^{k+1})^{2}
 \\&
 \reff{axiom:energy}\le 
 (1-\eps) \, \dist(u^\star, u_\ell^\star)^{2}
  + \eps \, C_1 \, \eta_\ell(u_\ell^{k+1})^2 + (1 + \eps \, C_2) \, \qctr^2 \, \dist(u_\ell^\star, u_\ell^k)^{2}
\end{align*}
According to the definition of $\QQ$ and Lemma~\ref{lemma:lambda}(ii), it holds that $k + 1 < \k(\ell)$ and hence
\begin{align*}
 \eta_\ell(u_\ell^{k+1})^2 
 \reff{eq2:stopping:linearization}< \lctr^{-2} \, \dist(u_\ell^{k+1}, u_\ell^k)^2
 \stackrel{\rm(ii)}\le \lctr^{-2} \, (1+\qctr)^2 \, \dist(u_\ell^\star, u_\ell^k)^2.
\end{align*}
Let $C_3 := \lctr^{-2} \, (1+\qctr)^2$. Combining the latter two estimates, we see that
\begin{align*}
 \Lambda_\ell^{k+1} 
 &= \dist(u^\star, u_\ell^{k+1})^{2} + \mu \, \eta_\ell(u_\ell^{k+1})^2 
 \\& 
 \le (1-\eps) \, \dist(u^\star, u_\ell^\star)^{2}
  + (\mu + \eps \, C_1) \, \eta_\ell(u_\ell^{k+1})^2 + (1 + \eps \, C_2) \, \qctr^2 \, \dist(u_\ell^\star, u_\ell^k)^{2}
 \\& 
 \le
 (1-\eps) \, \dist(u^\star, u_\ell^\star)^{2}
 + \big\{ (\mu + \eps \, C_1) \, C_3 + (1 + \eps \, C_2) \, \qctr^2 \big\} \, \dist(u_\ell^\star, u_\ell^k)^{2}
\end{align*}
Note that $C_1, C_2, C_3$ depend only on the problem setting. Provided that
\begin{align}\label{eq:prop:lambda:i}
 (\mu + \eps \, C_1) \, C_3 + (1 + \eps \, C_2) \, \qctr^2 \le 1-\eps,
\end{align}
we are thus led to 
\begin{align*}
 \Lambda_\ell^{k+1} \le (1-\eps) \, \big( \dist(u^\star, u_\ell^\star)^{2} + \dist(u_\ell^\star, u_\ell^k)^{2} \big)
 \reff{eq:galerkin}= (1-\eps) \, \dist(u^\star, u_\ell^k)^{2}
 \le (1-\eps) \, \Lambda_\ell^k.
\end{align*}
Up to the final choice of $\mu, \eps > 0$ (see below), this concludes the proof of Lemma~\ref{prop:lambda}(i).
\end{proof}

\begin{proof}[Proof of Lemma~\ref{prop:lambda}\rm(ii)]
Let $\mu, \delta, \eps > 0$ be free parameters, which will be fixed below. First, we note that
\begin{align*}
 \enorm{u^\star - u_\ell^\star}^2
 \reff{axiom:reliability}\le \Crel^2 \, \eta_\ell(u_\ell^\star)^2
 \reff{axiom:stability}\le 2 \, \Crel^2 \, \eta_\ell(u_\ell^{\k-1})^2 + 2\, \Crel^2 \Cstab^2 \, \enorm{u_\ell^\star - u_\ell^{\k -1}}^2.
\end{align*}
Together with~\eqref{eq:energy}, this leads to
\begin{align*}
 \dist(u^\star, u_\ell^\star)^{2}
 \reff{eq:energy}\le \frac{L}{2} \, \enorm{u^\star - u_\ell^\star}^2
 \reff{eq:energy}\le L \, \Crel^2 \, \eta_\ell(u_\ell^{\k-1})^2 
   + 2 L \alpha^{-1} \, \Crel^2 \Cstab^2 \, \dist(u_\ell^\star, u_\ell^{\k -1})^{2}
\end{align*}
Recall that $C_1 = L \, \Crel^2$ and $C_2 = 2 L \alpha^{-1} \, \Crel^2 \Cstab^2$. With this, we obtain that
\begin{align}\label{eq:lambda:ii:3}
\begin{split}
 \dist(u^\star, u_\ell^\k)^{2}
 &\reff{eq:galerkin}= (1-\eps) \, \dist(u^\star, u_\ell^\star)^{2} + \eps \, \dist(u^\star, u_\ell^\star)^{2} + \dist(u_\ell^\star, u_\ell^\k)^{2}
 \\& 
  \, \le \, (1-\eps) \, \dist(u^\star, u_\ell^\star)^{2}
   + \eps \, C_1 \, \eta_\ell(u_\ell^{\k-1})^2 
   + \eps \, C_2 \, \dist(u_\ell^\star, u_\ell^{\k -1})^{2}
   + \dist(u_\ell^\star, u_\ell^\k)^{2}
  \\& 
  \reff{axiom:energy}\le (1-\eps) \, \dist(u^\star, u_\ell^\star)^{2}
   + \eps \, C_1 \, \eta_\ell(u_\ell^{\k-1})^2 
   + ( \eps \, C_2 + \qctr^2 ) \, \dist(u_\ell^\star, u_\ell^{\k -1})^{2}.
\end{split}
\end{align}
Next, stability~\eqref{axiom:stability} and reduction~\eqref{axiom:reduction} show that
\begin{align*}
 \eta_{\ell+1}(u_\ell^\k)^2
 &= \eta_{\ell+1}(\TT_\ell \cap \TT_{\ell+1}, u_\ell^\k)^2
 + \eta_{\ell+1}(\TT_{\ell+1} \backslash \TT_\ell, u_\ell^\k)^2
 \\&
 \le \eta_{\ell}(\TT_\ell \cap \TT_{\ell+1}, u_\ell^\k)^2
 + \qred^2 \eta_{\ell}(\TT_\ell \backslash \TT_{\ell+1}, u_\ell^\k)^2
 \\&
 = \eta_\ell(u_\ell^\k)^2 - (1 - \qred^2)\, \eta_{\ell}(\TT_\ell \backslash \TT_{\ell+1}, u_\ell^\k)^2.
\end{align*}
According to the D\"orfler marking criterion in Algorithm~\ref{algorithm}(v), we are led to
\begin{align}\label{eq:lambda:ii:1}
 \eta_{\ell+1}(u_\ell^\k)^2
 \le \big( 1 - (1 - \qred^2)\,\theta^2 \big) \, \eta_\ell(u_\ell^\k)^2 =: q_\theta \, \eta_\ell(u_\ell^\k)^2.
\end{align}
Note that
\begin{align*}
 \enorm{u_\ell^\k - u_\ell^{\k -1}}^2
 \le 2 \, \big( \enorm{u_\ell^\star - u_\ell^\k}^2 + \enorm{u_\ell^\star - u_\ell^{\k -1}}^2 \big)
 &\reff{eq:energy}\le \frac{4}{\alpha} \, \big( \dist(u_\ell^\star, u_\ell^\k)^2 + \dist(u_\ell^\star, u_\ell^{\k -1})^2 \big)
 \\&
 \reff{axiom:energy}\le \frac{4}{\alpha} \, (\qctr^2 + 1) \, \dist(u_\ell^\star, u_\ell^{\k -1})^2.
\end{align*}
The Young inequality proves that
\begin{align}\label{eq:lambda:ii:2}
 \begin{split}
 \eta_\ell(u_\ell^\k)^2
 &\reff{axiom:stability}\le (1+\delta) \, \eta_\ell(u_\ell^{\k-1})^2 
   + (1+\delta^{-1}) \, \Cstab^2 \, \enorm{u_\ell^\k - u_\ell^{\k -1}}^2
 \\&
 \,\, \le\, (1+\delta) \, \eta_\ell(u_\ell^{\k-1})^2 
   + (1+\delta^{-1}) \, \frac{4}{\alpha} \, (\qctr^2 + 1) \, \Cstab^2 \, \dist(u_\ell^\star, u_\ell^{\k -1})^2.
 \end{split}
\end{align}
Let $C_4 := 4 \, \alpha^{-1} \, (\qctr^2 + 1) \,  \Cstab^2$. Note that Algorithm~\ref{algorithm} guarantees that $u_{\ell+1}^0 = u_\ell^\k$. Combining the latter estimates, we see that
\begin{align*}
 &\Lambda_{\ell+1}^0 
 = \dist(u^\star, u_{\ell+1}^0)^{2} + \mu \, \eta_{\ell+1}(u_{\ell+1}^0)^2
 \reff{eq:lambda:ii:1}\le \dist(u^\star, u_\ell^\k)^{2} + \mu \, q_\theta \, \eta_\ell(u_\ell^\k)^2
 \\& \quad
 \reff{eq:lambda:ii:3}\le (1-\eps) \, \dist(u^\star, u_\ell^\star)^{2} 
   + \eps \, C_1 \, \eta_\ell(u_\ell^{\k-1})^2 
   + (\eps \, C_2 + \qctr^2 ) \, \dist(u_\ell^\star, u_\ell^{\k -1})^{2}
   + \mu \, q_\theta \, \eta_\ell(u_\ell^\k)^2
 \\& \quad
 \reff{eq:lambda:ii:2}\le (1-\eps) \, \dist(u^\star, u_\ell^\star)^{2}
   + \big\{ \eps \, C_1 \, \mu^{-1} + q_\theta \, (1 + \delta) \big\} \, \mu \, \eta_\ell(u_\ell^{\k-1})^2 
 \\& \qquad \qquad + \big\{ \eps \, C_2 + \qctr^2 + \mu \, q_\theta \, (1 + \delta^{-1}) C_4 \big\} \, \dist(u_\ell^\star, u_\ell^{\k -1})^{2}. 
\end{align*}
Note that $C_1, C_2, C_4$ and $0 < q_\theta < 1$ depend only on the problem setting. Provided that 
\begin{align}\label{eq:prop:lambda:ii}
 \eps \, C_1 \, \mu^{-1} + q_\theta \, (1 + \delta) \le 1-\eps
 \quad \text{and} \quad
  \eps \, C_2 + \qctr^2 + \mu \, q_\theta \, (1 + \delta^{-1}) C_4 \le 1-\eps,
\end{align}
we are thus led to 
\begin{align*}
 &\Lambda_{\ell+1}^0 
 \le (1-\eps) \, \big( \dist(u^\star, u_\ell^\star)^{2} + \dist(u_\ell^\star, u_\ell^{\k-1})^{2} + \mu \, \eta_\ell(u_\ell^{\k-1})^2 \big)
 \reff{eq:galerkin}= 
 (1-\eps) \, \Lambda_\ell^{\k-1}.
\end{align*}
Up to the final choice of $\delta, \mu, \eps > 0$, this concludes the proof of Lemma~\ref{prop:lambda}(ii).
\end{proof}

\begin{proof}[Proof of Lemma~\ref{prop:lambda} (fixing the free parameters)]
We proceed as follows:
\begin{itemize}
 \item Choose $\delta > 0$ such that \ $(1+\delta) \, q_\theta < 1$.
 \item Choose $\mu > 0$ such that \ $\qctr^2 + \mu \, q_\theta(1+\delta)^{-1} \, C_4 < 1$ \ and \ $\mu C_3 + \qctr^2 < 1$.
 \item Finally, choose $\eps > 0$ sufficiently small such that~\eqref{eq:prop:lambda:i} and~\eqref{eq:prop:lambda:ii} are satisfied.
\end{itemize}
This concludes the proof of Lemma~\ref{prop:lambda} with $(1-\eps) = \qlin^2.$
\end{proof}

\begin{proof}[\textbf{Proof of Theorem~\ref{theorem:linconv} under assumption~\bf(\ref{axiom:energy})}]
According to~\eqref{eq:energy}, it holds that 
$\Delta_\ell^k \simeq (\Lambda_\ell^k)^{1/2}$,
where the hidden constants depend only on $\mu$, $\alpha$, and $L$. Then, linear convergence~\eqref{eq:theorem:linconv} follows from Lemma~\ref{prop:lambda} and induction, since the set $\QQ$ is linearly ordered with respect to the total step counter $|(\cdot,\cdot)|$.
\end{proof}

\subsection{Proof of Theorem~\ref{theorem:linconv} under assumption~(\ref{axiom:norm})}
We recall the following main result from~\cite{banach}. The proof is based on a perturbation argument. It is shown that the given constraint on $\lctr$ guarantees estimator equivalence $\eta_\ell(u_\ell^\star) \simeq \eta_\ell(u_\ell^\k)$ as well as the fact that D\"orfler marking for  $\eta_\ell(u_\ell^\k)$ and $\theta$ implies D\"orfler marking for $\eta_\ell(u_\ell^\star)$ and $\theta^\star := (\theta - \lctr/\lconv)/(1 + \lctr/\lconv) > 0$.

\begin{lemma}[{{\cite[Lemma~4.9, Theorem~5.3]{banach}}}]
Suppose \eqref{axiom:stability}--\eqref{axiom:reliability} and \eqref{axiom:norm}.
Let $0 < \theta \leq 1$ and $0 < \lctr < \lconv \theta$, where $\lconv = \frac{1-\qctr}{\Cstab\qctr}$. Then, 
it holds that
\begin{align}\label{eq:eta-star}
 (1 - \lctr/\lconv) \, \eta_\ell(u_\ell^\k) 
 \le \eta_\ell(u_\ell^\star) 
 \le (1 + \lctr/\lconv) \, \eta_\ell(u_\ell^\k).
\end{align}
Moreover, there exist $\Cghps > 0$ and $0 < \qghps < 1$ such that
\begin{align}\label{eq:linearconvergence}
 \eta_{\ell + n}(u_{\ell+n}^\k) 
 \le \Cghps \, \qghps^n \, \eta_\ell(u_\ell^\k)
 \quad \text{for all } (\ell + n +1, 0)\in\QQ.
\end{align}
The constants $\Cghps$ and $\qghps$ depend only on $\Ccea=L/\alpha$, $\Crel$, $\Cstab$, $\qred$, and $\qctr$, as well as on the adaptivity parameters $\theta$ and $\lctr$. \qed
\end{lemma}

In the present case, the core of the proof is the following summability result.

\begin{lemma}\label{lemma:linconv:sum}
Suppose \eqref{axiom:stability}--\eqref{axiom:reliability} and \eqref{axiom:norm}. 
Let $0 < \theta \leq 1$ and $0 < \lctr < \lconv \theta$, where again $\lconv = \frac{1-\qctr}{\qctr\Cstab}$. Then, there exists $\Csum > 0$ such that
\begin{align}\label{eq:ell-k}
 \sum_{\substack{(\ell,k) \in \QQ \\ (\ell,k) > (\ell',k')}} \Delta_\ell^k
 \le \Csum \, \Delta_{\ell'}^{k'}\quad\textrm{for all }(\ell',k') \in \QQ.
\end{align}
The constant $\Csum$ depends only on $L$, $\alpha$, $\Crel$, $\Cstab$, $\qred$, and $\qctr$, as well as on the adaptivity parameters $\theta$ and $\lctr$.
\end{lemma}

\begin{proof}
\def\boxed#1{#1}
The proof is split into six steps. 

{\bf Step~1.}  This step provides an equivalent quasi-error quantity. 
First, note that
\begin{align*}
 \enorm{u^\star - u_\ell^k}
 \le \enorm{u^\star - u_\ell^\star} + \enorm{u_\ell^\star - u_\ell^k}
 \reff{axiom:reliability}\lesssim \eta_\ell(u_\ell^\star) + \enorm{u_\ell^\star - u_\ell^k}
 \reff{axiom:stability}\lesssim \eta_\ell(u_\ell^k) + \enorm{u_\ell^\star - u_\ell^k}
 =: \Alpha_\ell^k.
\end{align*}
This proves that $\Delta_\ell^k = \enorm{u^\star - u_\ell^k} + \eta_\ell(u_\ell^k) \lesssim \Alpha_\ell^k$.
Second, the C\'ea lemma~\eqref{eq:cea} proves that 
\begin{align*}
 \enorm{u_\ell^\star - u_\ell^k}
 \le  \enorm{u^\star - u_\ell^\star} + \enorm{u^\star - u_\ell^k} 
 \reff{eq:cea}\lesssim \enorm{u^\star - u_\ell^k}.
\end{align*}
This concludes that 
 \begin{align}\label{eq:aux1}
  \boxed{\Alpha_\ell^k = \enorm{u_\ell^\star - u_\ell^k} + \eta_\ell(u_\ell^k)
  \simeq \Delta_\ell^k.}
 \end{align}

{\bf Step~2.} This step collects some auxiliary estimates. We start with
\begin{align}\label{eq1:ell-k}
 \boxed{\Alpha_\ell^0 \lesssim \eta_{\ell-1}(u_{\ell-1}^\k) \le \Alpha_{\ell-1}^\k 
 \quad \text{for all } (\ell, 0) \in \QQ \text{ with } \ell > 0.}
\end{align}
With the C\'ea lemma~\eqref{eq:cea} and reliability~\eqref{eq:lemma:reliable}, it follows that
\begin{align*} 
 & \enorm{u_\ell^\star - u_{\ell-1}^\k}
 \le \enorm{u^\star - u_\ell^\star} + \enorm{u^\star - u_{\ell-1}^\k}
 \reff{eq:cea}\lesssim \enorm{u^\star - u_{\ell-1}^\k}
 \reff{eq:lemma:reliable}\lesssim \eta_{\ell-1}(u_{\ell-1}^\k)
\end{align*}
With nested iteration $u_\ell^0 = u_{\ell-1}^\k$ and~\eqref{axiom:stability}--\eqref{axiom:reduction}, we thus obtain that
\begin{align*}
 \Alpha_\ell^0 = \enorm{u_\ell^\star - u_\ell^0} + \eta_\ell(u_\ell^0)
 = \enorm{u_\ell^\star - u_{\ell-1}^\k} + \eta_\ell(u_{\ell-1}^\k)
 \lesssim 
 \eta_{\ell-1}(u_{\ell-1}^\k) \le \Alpha_{\ell-1}^\k
\end{align*}
This proves~\eqref{eq1:ell-k}. Next, we prove that
\begin{align}\label{eq2:ell-k}
 \boxed{\Alpha_\ell^\k \lesssim \Alpha_\ell^k
 \quad \text{for all } (\ell+1, 0) \in \QQ
 \text{ and } 0 \le k \le \k(\ell).}
\end{align}
To see this, note that
\begin{align*}
 \enorm{u_\ell^\k - u_\ell^k} 
 \le \enorm{u_\ell^\star - u_\ell^\k} + \enorm{u_\ell^\star - u_\ell^k}
 \reff{axiom:norm}\le (\qctr^{\k-k}\ + 1) \, \enorm{u_\ell^\star - u_\ell^k}.
\end{align*}
Hence, it follows that
\begin{align*}
 \Alpha_\ell^\k = \enorm{u_\ell^\star \!-\! u_\ell^\k} + \eta_\ell(u_\ell^\k)
 &\reff{axiom:stability}\lesssim  \enorm{u_\ell^\star \!-\! u_\ell^k} + \enorm{u_\ell^\k \!-\! u_\ell^k} +\eta_\ell(u_\ell^k)
 \,\,\lesssim\,\, \enorm{u_\ell^\star - u_\ell^k} + \eta_\ell(u_\ell^k)
 = \Alpha_\ell^k.
\end{align*}
This proves~\eqref{eq2:ell-k}. Finally, we use the stopping criterion of Algorithm \ref{algorithm}(iv) to prove that
\begin{align}\label{eq3:ell-k}
 \boxed{\Alpha_\ell^k \lesssim \enorm{u_\ell^\star - u_\ell^{k-1}}
 \quad \text{for all } (\ell,k) \in \QQ \text{ with } k > 0.}
\end{align}
With the stopping criterion~\eqref{eq2:stopping:linearization} and Lemma~\ref{lemma:lambda}(ii), we get that
\begin{align*}
 \eta_\ell(u_\ell^k)
\reff{eq2:stopping:linearization}\lesssim \enorm{u_\ell^k - u_\ell^{k-1}}
 \stackrel{\rm(ii)}\lesssim \enorm{u_\ell^\star - u_\ell^{k-1}}.
\end{align*}
This leads to
\begin{align*}
 \Alpha_\ell^k = \enorm{u_\ell^\star - u_\ell^k} + \eta_\ell(u_\ell^k)
 \reff{axiom:norm}\lesssim \enorm{u_\ell^\star - u_\ell^{k-1}}
\end{align*}
and thus proves~\eqref{eq3:ell-k}.

{\bf Step~3.} Suppose that $\underline\ell = \infty$ and hence $\k(\ell) < \infty$ for all $\ell \in \N_0$. Note that
\begin{align*}
 &\sum_{\substack{(\ell,k) \in \QQ \\ (\ell,k) > (\ell',k')}} \Alpha_\ell^k
 = \sum_{\ell = \ell' + 1}^\infty \sum_{k = 0}^{\k(\ell)-1} \Alpha_\ell^k
 + \sum_{k = k' + 1}^{\k(\ell')-1} \Alpha_{\ell'}^k
 \reff{eq1:ell-k}\lesssim \sum_{\ell = \ell' + 1}^\infty \sum_{k = 1}^{\k(\ell)} \Alpha_\ell^k
 + \sum_{k = k' + 1}^{\k(\ell')} \Alpha_{\ell'}^k.
\end{align*}
With contraction~\eqref{axiom:norm}, the geometric series proves that
\begin{align}\label{eq:sum:k}
 \sum_{k = i + 1}^{\k(\ell)-1} \Alpha_\ell^k
 \reff{eq3:ell-k}\lesssim \sum_{k = i + 1}^{\k(\ell)-1} \enorm{u_\ell^\star - u_\ell^{k-1}}
 \reff{axiom:norm}\le \enorm{u_\ell^\star - u_\ell^i} \, \sum_{k = i}^{\infty} \qctr^{k-i}
\lesssim \Alpha_\ell^i
 \quad \text{for all } (\ell,i) \in \QQ.
\end{align}
Hence, it follows that
\begin{align*}
 \sum_{k = 1}^{\k(\ell)}A_\ell^k 
 \quad
 \begin{cases}
 \quad\,=\, \Alpha_\ell^1 \reff{eq2:ell-k} \lesssim A_\ell^0
 \quad& \text{ if } \k(\ell) = 1,\\
 \quad\displaystyle\reff{eq2:ell-k}\lesssim  \sum_{k = 1}^{\k(\ell)-1} \Alpha_\ell^k
 \reff{eq:sum:k}\lesssim A_\ell^0
 \quad& \text{ if } \k(\ell) > 1.
 \end{cases}
\end{align*}
Moreover, it follows that
\begin{align*}
 \sum_{k=k'+1}^{\k(\ell')} A_{\ell'}^k 
 \quad
 \begin{cases}
 \quad\,=\, A_{\ell'}^{k'+1} \reff{eq2:ell-k} \lesssim A_{\ell'}^{k'}
 \quad& \text{ if } \k(\ell') = k' + 1,\\
 \quad\displaystyle \reff{eq2:ell-k} \lesssim \sum_{k = k' + 1}^{\k(\ell')-1} \Alpha_{\ell'}^k
 \reff{eq:sum:k} \lesssim A_{\ell'}^{k'}
 \quad& \text{ if } \k(\ell') > k' + 1.
 \end{cases}
\end{align*}
So far, this proves that
\begin{align*}
 \sum_{\substack{(\ell,k) \in \QQ \\ (\ell,k) > (\ell',k')}} \Alpha_\ell^k
 \lesssim \Alpha_{\ell'}^{k'} + \sum_{\ell = \ell' + 1}^\infty \Alpha_\ell^0.
\end{align*}
Exploiting the linear convergence~\eqref{eq:linearconvergence} together with the geometric series, we prove that
\begin{align*}
 \sum_{\ell = \ell' + 1}^\infty \Alpha_\ell^0
 \reff{eq1:ell-k}\lesssim \sum_{\ell = \ell' + 1}^\infty \eta_{\ell-1}(u_{\ell-1}^\k)
 = \sum_{\ell = \ell'}^\infty \eta_{\ell}(u_{\ell}^\k)
 \reff{eq:linearconvergence}\lesssim \eta_{\ell'}(u_{\ell'}^\k) \sum_{\ell = \ell'}^\infty \qghps^{\ell-\ell'}
 \simeq \eta_{\ell'}(u_{\ell'}^\k)
 \le \Alpha_{\ell'}^\k.
\end{align*}
Overall, this proves that
\begin{align}
\label{eq:sum Alk1}
&\sum_{\substack{(\ell,k) \in \QQ \\ (\ell,k) > (\ell',k')}}  \Alpha_\ell^k
\lesssim \Alpha_{\ell'}^{k'}+ \Alpha_{\ell'}^\k  
 \reff{eq2:ell-k}\simeq \Alpha_{\ell'}^{k'}
 \quad \text{provided that } \underline\ell=\infty.
\end{align}

{\bf Step~4.} Suppose that $\ell' = \underline\ell < \infty$ and hence  $\k(\ell') = \k(\underline\ell) = \infty$.
Then, the geometric series proves that
\begin{align}\label{eq:sum Alk2}
 \sum_{\substack{(\ell,k) \in \QQ \\ (\ell,k) > (\ell',k')}}  \Alpha_\ell^k
 = \sum_{k = k'+1}^{\infty} \Alpha_{\ell'}^k
 \reff{eq:sum:k}\lesssim 
 \Alpha_{\ell'}^{k'}.
\end{align}

{\bf Step~5.} Suppose that $\ell' < \underline\ell < \infty$ and hence $\k(\underline\ell) = \infty$. Then, it holds that
\begin{align*}
\sum_{\substack{(\ell,k) \in \QQ \\ (\ell,k) > (\ell',k')}}  \Alpha_\ell^k
& = \sum_{\ell=\ell'+1}^{\underline\ell-1}\sum_{k=0}^{\k(\ell)-1}\Alpha_\ell^k + \sum_{k = k'+1}^{\k(\ell')}-1 \Alpha_{\ell'}^k + \sum_{k=0}^\infty\Alpha_{\underline\ell}^k.
\end{align*}
First, note that
\begin{align*}
\sum_{k=0}^\infty\Alpha_{\underline\ell}^k 
= \Alpha_{\underline\ell}^0 + \sum_{k=1}^\infty\Alpha_{\underline\ell}^k 
\reff{eq:sum:k}\leq \Alpha_{\underline\ell}^0
\reff{eq1:ell-k}\lesssim\Alpha_{\underline\ell-1}^\k.
\end{align*}
Provided that $\ell' < \underline\ell < \infty$, it hence holds that
\begin{align*}
&\sum_{\substack{(\ell,k) \in \QQ \\ (\ell,k) > (\ell',k')}}  \Alpha_\ell^k
\lesssim \sum_{\ell=\ell'+1}^{\underline\ell-1}\sum_{k=0}^{\k(\ell)}\Alpha_\ell^k + \sum_{k = k'+1}^{\k(\ell')-1} \Alpha_{\ell'}^k
\reff{eq1:ell-k}\lesssim \sum_{\ell=\ell'+1}^{\underline\ell-1}\sum_{k=1}^{\k(\ell)}\Alpha_\ell^k + \sum_{k = k'+1}^{\k(\ell')} \Alpha_{\ell'}^k.
\end{align*}
Along the lines of Step~3, one concludes that 
\begin{align}\label{eq:sum Alk3}
 \sum_{\ell=\ell'+1}^{\underline\ell-1}\sum_{k=1}^{\k(\ell)}\Alpha_\ell^k + \sum_{k = k'+1}^{\k(\ell')} \Alpha_{\ell'}^k
\lesssim \Alpha_{\ell'}^{k'}.
\end{align}

{\bf Step~6.} In any case, \eqref{eq:sum Alk1}--\eqref{eq:sum Alk3} prove that
$$
 \sum_{\substack{(\ell,k) \in \QQ \\ (\ell,k) > (\ell',k')}} \Delta_\ell^k 
 \simeq \sum_{\substack{(\ell,k) \in \QQ \\ (\ell,k) > (\ell',k')}} \Alpha_\ell^k 
 \lesssim \Alpha_{\ell'}^{k'}
 \simeq \Delta_{\ell'}^{k'}
 \quad\textrm{for all }(\ell',k') \in \QQ.
$$
This concludes the proof of~\eqref{eq:ell-k}.
\end{proof}

\begin{proof}[{\bfseries Proof of Theorem~\ref{theorem:linconv} under the assumption~\bf(\ref{axiom:norm})}]

From \cite[Lemma~4.9]{axioms}, we recall the following implication for sequences $(\alpha_n)_{n \in \N_0}$ in $\R_{\ge0}$ and constants $C > 0$:
\begin{align*}
 \bigg[\forall N \in \N_0: \, \sum_{n = N + 1}^\infty \alpha_n \le C \, \alpha_N\bigg]
 \,\,\, \Longrightarrow \,\,\,
 \bigg[\forall N,m \in \N_0: \, \alpha_{N+m} \le (1 + C) (1 + C^{-1})^{-m} \alpha_N \bigg].
\end{align*}%
Since the index set $\QQ$ is linearly ordered with respect to the total step counter $|(\cdot,\cdot)|$, Lemma~\ref{lemma:linconv:sum} 
implies that
\begin{align*}
\Delta_{\ell^\prime}^{k^\prime}\leq\Clin\,\qlin^{|(\ell^\prime,k^\prime)|-|(\ell,k)|}\,\Delta_\ell^k\quad\textrm{for all }(\ell,k),(\ell^\prime,k^\prime)\in\QQ\textrm{ with }(\ell^\prime,k^\prime) > (\ell,k),
\end{align*}
where $\Clin = 1 + \Csum$ and $\qlin = 1 / (1 + \Csum^{-1})$. 
\end{proof}

\section{Proof of Theorem~\ref{theorem:rates} (optimal convergence rates)}
\label{section:costs}

Recall $\norm{u^\star}{\A_s}$ from \eqref{eq:As} and $\T(N)=\set{\TT\in\refine(\TT_0}{\#\TT-\#\TT_0\le N}$.
The following lemma proves the first inequality in \eqref{eq:optimal}.
\begin{lemma}\label{lemma:opt:lower}
Suppose~\eqref{axiom:sons} as well as~\eqref{axiom:stability}--\eqref{axiom:reliability}.
Let $s > 0$. 
Then, it holds that 
\begin{align}
 \norm{u^\star}{\A_s}
 \le \copt \, \sup_{(\ell,k) \in \QQ} (\#\TT_\ell - \#\TT_0 + 1)^s \Delta_\ell^k,
\end{align}
where  $\copt > 0$ depends only on $\Ccea=L/\alpha$, $\Cson$, $\Cstab$, $\Crel$, $\#\TT_0$, and $s$, 
and, if $\underline\ell<\infty$ or $\eta_{\ell_0}(u_{\ell_0}^{\underline k})=0$ for some $(\ell_0+1,0)\in\QQ$, additionally on $\underline\ell$ resp.\ $\ell_0$.
\end{lemma}

\begin{proof}
The proof is split into three steps.
First, we recall from~\cite[Lemma~22]{bhp2017} that
\begin{align}\label{eq:TT0}
 \#\TT_\fine / \#\TT_\coarse
 \le \#\TT_\fine - \#\TT_\coarse + 1 
 \le \#\TT_\fine
 \quad \text{for all } \TT_\coarse \in \T \text{ and all } \TT_\fine \in \refine(\TT_\coarse).
\end{align}

{\bf Step~1.}
In this step, we consider the pathological cases that $\underline\ell < \infty$ or $\eta_{\ell_0}(u_{\ell_0}^{\underline k})=0$ for some $(\ell_0+1,0)\in\QQ$.
In the first case, Corollary~\ref{cor:linconv} gives that $u^\star=u_{\underline\ell}^\star$ as well as $\eta_{\underline\ell}(u_{\underline\ell}^\star)=0$.
Note that the latter implies that $u_{\ell_0}^{\underline k}=u^\star=u_{\ell_0}^\star$ due to Proposition~\ref{lemma:reliable} and \eqref{eq:cea}. 
Hence, with $\ell':=\underline \ell$ resp.\ $\ell':=\ell_0$, we obtain that 
\begin{align*}
\norm{u^\star}{\A_s} = \max_{0 \le N < \#\TT_{\ell'}-\#\TT_0}(N+1)^s \min_{\TT_\coarse \in \T(N)} \big( \enorm{u^\star - u_\coarse^\star} + \eta_\coarse(u_\coarse^\star) \big).
\end{align*}
The term $N+1$ within the maximum can be estimated by
\begin{align*}
N+1\le \#\TT_{\ell'}-\#\TT_0\reff{axiom:sons}
\le (\Cson^{\ell'}-1)\,\#\TT_0.
\end{align*}
The C\'ea lemma \eqref{eq:cea} and \eqref{axiom:stability}--\eqref{axiom:reliability} give that $\enorm{u^\star-u_H^\star}\lesssim \enorm{u^\star-u_0^\star}$ and  $\eta_H(u^\star_H)\lesssim \eta_0(u^\star_0)$ (see, e.g.,  \cite[Lemma~3.5]{axioms}). 
Altogether, we thus arrive at
\begin{align}\label{eq:upper bound for As1}
\norm{u^\star}{\A_s} 
\lesssim 
\big( \enorm{u^\star - u_0^\star} + \eta_0(u_0^\star) \big).
\end{align}

{\bf Step~2.}
Next, we consider the generic case that $\underline\ell = \infty$ and $\eta_{\ell_0}(u_{\ell_0}^{\underline k})>0$ for all $\ell_0\in\N_0$.
Algorithm~\ref{algorithm} yields that $\#\TT_\ell\to\infty$ as $\ell\to\infty$.
Thus, we can argue analogously to the proof of~\cite[Theorem~4.1]{axioms}:
Let $N \in \N_0$. 
Choose the maximal $\ell \in \N_0$ such that 
$ \#\TT_\ell - \#\TT_0 + 1 \le N$.
Then, $\TT_\ell \in \T(N)$. 
The choice of $N$ guarantees that
\begin{align}
\label{eq:upper bound for As2}
N+1\le \#\TT_{\ell+1} - \#\TT_0 + 1 
 \reff{eq:TT0}\le \#\TT_{\ell+1}
 \le \Cson \#\TT_\ell
 \reff{eq:TT0}\le \Cson \#\TT_0 \, ( \#\TT_\ell - \#\TT_0 + 1 ).
\end{align}
This leads to 
$$
 (N+1)^s \min_{\TT_\coarse \in \T(N)} \big( \enorm{u^\star - u_\coarse^\star} + \eta_\coarse(u_\coarse^\star) \big)
 \lesssim (\#\TT_\ell - \#\TT_0 + 1)^s \big( \enorm{u^\star - u_\ell^\star} + \eta_\ell(u_\ell^\star) \big).
$$
Taking the supremum over all $N \in \N_0$, we conclude that
\begin{align}
\label{eq:upper bound for As3}
 \norm{u^\star}{\A_s} 
 \lesssim \sup_{\ell \in \N_0}(\#\TT_\ell - \#\TT_0 + 1)^s \big( \enorm{u^\star - u_\ell^\star} + \eta_\ell(u_\ell^\star) \big).
\end{align}

{\bf Step~3.} With stability~\eqref{axiom:stability} and the C\'ea lemma~\eqref{eq:cea}, we see for all $(\ell,0)\in\QQ$ that
\begin{align*}
 &\enorm{u^\star - u_\ell^\star} + \eta_\ell(u_\ell^\star)
 \reff{axiom:stability} \lesssim \enorm{u^\star - u_\ell^\star} + \enorm{u_\ell^\star - u_\ell^0} + \eta_\ell(u_\ell^0)
 \\ &\quad
 \le 2 \, \enorm{u^\star - u_\ell^\star} + \enorm{u^\star - u_\ell^0} + \eta_\ell(u_\ell^0)
 \reff{eq:cea} \lesssim \enorm{u^\star - u_\ell^0} + \eta_\ell(u_\ell^0)
 = \Delta_\ell^0.
\end{align*}
With \eqref{eq:upper bound for As1} and \eqref{eq:upper bound for As3}, we thus obtain  that
$$
 \norm{u^\star}{\A_s}
 \lesssim \sup_{(\ell,0) \in \QQ} (\#\TT_\ell - \#\TT_0 + 1)^s \, \big(\enorm{u^\star-u_\ell^\star} +\eta_\ell(u_\ell^\star)\big)
 \le \sup_{(\ell, k) \in \QQ} (\#\TT_\ell - \#\TT_0 + 1)^s \, \Delta_\ell^k.
$$
This concludes the proof.
\end{proof}

To prove the converse estimate, we need the following comparison lemma for the error estimator of the exact discrete solution $u_\ell^\star \in \XX_\ell$, which is found in \cite[Lemma~4.14]{axioms}.

\begin{lemma}\label{lemma:doerfler2}
Suppose \eqref{axiom:sons}--\eqref{axiom:overlay} and \eqref{axiom:stability}--\eqref{axiom:discrete_reliability}. Let $0 < \theta' < \theta_\opt := (1 + \Cstab^2\Crel^2)^{-1/2}$.
Then, there exist constants $C_1, C_2 > 0$ such that for all $s>0$ with $\norm{u^\star}{\mathbb{A}_s} < \infty$ and all $\TT_\coarse \in \T$, there exists $\RR_\coarse \subseteq \TT_\coarse $ which satisfies that
\begin{align}\label{eq:doerfler2_Rbound}
\# \mathcal{R}_\coarse \leq C_1 C_2^{-1/s} \, \norm{u^\star}{\mathbb{A}_s}^{1/s} \eta_\coarse(u_\coarse^\star)^{-1/s},
\end{align}
and the D\"orfler marking criterion
\begin{align}\label{eq:doerfler2_doerfler}
\theta' \eta_\coarse(u_\coarse^\star) \leq \eta_\coarse(\mathcal{R}_\coarse, u_\coarse^\star).
\end{align}
The constants $C_1,C_2$ depend only on the constants of~\eqref{axiom:stability}--\eqref{axiom:discrete_reliability}. \qed
\end{lemma}

\begin{proof}[{\bfseries Proof of Theorem~\ref{theorem:rates}}]
The proof is split into six steps.

{\bf Step~1.}
It holds that
$$
 \sup_{(\ell',k') \in \QQ} (\#\TT_{\ell'} - \#\TT_0 + 1)^s \, \Delta_{\ell'}^{k'}
 \le \sup_{(\ell',k') \in \QQ} \bigg(\sum_{\substack{(\ell,k) \in \QQ \\ (\ell,k) \le (\ell',k')}} \#\TT_\ell \bigg)^s \, \Delta_{\ell'}^{k'}.
$$
According to Lemma~\ref{lemma:opt:lower}, it only remains to prove that
\begin{align}\label{eq:rates:goal}
 \sup_{(\ell',k') \in \QQ} \bigg(\sum_{\substack{(\ell,k) \in \QQ \\ (\ell,k) \le (\ell',k')}} \#\TT_\ell \bigg)^s \, \Delta_{\ell'}^{k'}
 \lesssim \max\big\{\norm{u^\star}{\A_s},\Delta_0^0\big\}.
\end{align}
Without loss of generality, we may assume that $\norm{u^\star}{\A_s} < \infty$. 

{\bf Step~2.} Provided that $(\ell+1,0) \in \QQ$ (and hence $\underline k(\ell)<\infty$) Lemma~\ref{lemma:lambda}(i)$\&$(iii) and Algorithm~\ref{algorithm}(iv) prove that
$$
\dist(u_\ell^\star, u_\ell^\k)
\le \frac{\qctr}{1 \!-\! \qctr} \, \dist(u_\ell^\k, u_\ell^{\k-1})
\le \frac{\qctr}{1 \!-\! \qctr} \, \lctr \, \eta_\ell(u_\ell^\k).
$$
Under~\eqref{axiom:norm}, this leads to
\begin{subequations}\label{eq:stability}
\begin{align}
 \enorm{u_\ell^\star - u_\ell^\k} 
 = \dist(u_\ell^\star, u_\ell^\k)
\le \frac{\qctr}{1-\qctr} \lctr\eta_\ell(u_\ell^\k)
\reff{def:lambda_opt}{\le} \Cstab^{-1} \, \lctr/\lopt \, \eta_\ell(u_\ell^\k).
\end{align}
Under~\eqref{axiom:energy}, this leads to
\begin{align}
\enorm{u_\ell^\star - u_\ell^\k}
\reff{eq:energy}\le \sqrt{2/\alpha} \, \dist(u_\ell^\star, u_\ell^\k)
\reff{def:lambda_opt}\le \Cstab^{-1} \, \lctr/\lopt \, \eta_\ell(u_\ell^\k).
\end{align}
\end{subequations}

{\bf Step~3.} With Step~2, we see that 
\begin{align*}
 \eta_\ell(u_\ell^\k)
 &\reff{axiom:stability}\le \eta_\ell(u_\ell^\star) + \Cstab \, \enorm{u_\ell^\star - u_\ell^\k}
 \reff{eq:stability}\le \eta_\ell(u_\ell^\star) + \lctr / \lopt \, \eta_\ell(u_\ell^\k),
\\
 \eta_\ell(u_\ell^\star)
 &\reff{axiom:stability}\le \eta_\ell(u_\ell^\k) + \Cstab \, \enorm{u_\ell^\star - u_\ell^\k}
 \reff{eq:stability}\le \eta_\ell(u_\ell^\k) + \lctr / \lopt \, \eta_\ell(u_\ell^\k).
\end{align*}
With $0 < \lctr / \lopt < 1$, this guarantees the equivalence
\begin{align}\label{eq:stability:**}
 (1 -  \lctr / \lopt) \, \eta_\ell(u_\ell^\k) \le \eta_\ell(u_\ell^\star)
 \le (1 + \lctr / \lopt) \, \eta_\ell(u_\ell^\k)
 \quad \text{for all } (\ell+1,0) \in \QQ.
\end{align}

{\bf Step~4.} 
Let $\RR_\ell \subseteq \TT_\ell$ be the subset from Lemma~\ref{lemma:doerfler2} with $\theta'$ from \eqref{eq:opt:theta'}. Note that 
\begin{align}\label{eq2:stability}
 \eta_\ell(\RR_\ell, u_\ell^\star) 
 \reff{axiom:stability} \le \eta_\ell(\RR_\ell, u_\ell^\k) + \Cstab \, \enorm{u_\ell^\star - u_\ell^\k}
 \reff{eq:stability} \le \eta_\ell(\RR_\ell, u_\ell^\k) + \lctr/\lambda_\opt \, \eta_\ell(u_\ell^\k).
\end{align}
This proves that
\begin{align}\label{eq3:stability}
 (1 - \lctr/\lambda_\opt) \, \theta' \, \eta_\ell(u_\ell^\k)
 \reff{eq:stability:**} \le \theta' \, \eta_\ell(u_\ell^\star)
 \reff{eq:doerfler2_doerfler}\le \eta_\ell(\RR_\ell, u_\ell^\star)
 \reff{eq2:stability}\le \eta_\ell(\RR_\ell, u_\ell^\k)+ \lctr/\lambda_\opt \, \eta_\ell(u_\ell^\k).
\end{align}
The choice of $\theta'$ in~\eqref{eq:opt:theta'} gives that $\theta = (1 - \lctr/\lambda_\opt) \, \theta' - \lctr/\lambda_\opt$. 
Thus, we obtain that
$$
 \theta \, \eta_\ell(u_\ell^\k)
  \reff{eq:opt:theta'}= \big( (1 - \lctr/\lambda_\opt) \, \theta' - \lctr / \lopt \big) \, \eta_\ell(u_\ell^\k)
 \reff{eq3:stability}\le \eta_\ell(\RR_\ell, u_\ell^\k).
$$
Hence, $\RR_\ell$ satisfies the D\"orfler marking criterion used in Algorithm~\ref{algorithm}(v). By (quasi-) minimality of $\MM_\ell$ in Algorithm~\ref{algorithm}(v), we  infer that
$$
 \#\MM_\ell \lesssim \#\RR_\ell 
 \reff{eq:doerfler2_Rbound}\lesssim \norm{u^\star}{\A_s}^{1/s} \, \eta_\ell(u_\ell^\star)^{-1/s}
 \reff{eq:stability:**}\simeq \norm{u^\star}{\A_s}^{1/s} \, \eta_\ell(u_\ell^\k)^{-1/s}.
$$
Nested iteration guarantees that $u_{\ell+1}^0 = u_\ell^\k$. Thus, reliability~\eqref{eq:lemma:reliable} and~\eqref{axiom:stability}--\eqref{axiom:reduction} lead to
$$ 
 \eta_\ell(u_\ell^\k) \reff{eq:lemma:reliable}\simeq \Delta_\ell^\k 
 = \enorm{u^\star - u_{\ell+1}^0} + \eta_\ell(u_{\ell+1}^0)
 \ge \enorm{u^\star - u_{\ell+1}^0} + \eta_{\ell+1}(u_{\ell+1}^0)
 = \Delta_{\ell+1}^0.
$$
Overall, we derive that
\begin{align}\label{eq:optimality:step3}
 \#\MM_\ell 
 \lesssim \norm{u^\star}{\A_s}^{1/s} \, \eta_\ell(u_\ell^\k)^{-1/s}
 \lesssim \norm{u^\star}{\A_s}^{1/s} \, (\Delta_{\ell+1}^0)^{-1/s}
 \quad \text{for all } (\ell+1,0) \in \QQ.
\end{align}
The hidden constant depends only on $\Cstab$, $\qred$, $\Crel$, $1-\lctr/\lopt$, $\Cmark$, $\Crel'$, and~$s$.

{\bf Step~5.} 
For $(\ell, k) \in \QQ$ with $\TT_\ell \neq \TT_0$, Step~4 and the closure estimate~\eqref{axiom:closure} lead to 
$$
 \#\TT_\ell - \#\TT_0 + 1 \simeq  \#\TT_\ell - \#\TT_0
 \reff{axiom:closure} {\lesssim} \sum_{n = 0}^{\ell-1} \#\MM_n
 \reff{eq:optimality:step3}\lesssim 
 \norm{u^\star}{\A_s}^{1/s} \, \sum_{n = 0}^\ell (\Delta_n^0)^{-1/s}.
$$
Replacing $\norm{u^\star}{\A_s}$ with $\max\{\norm{u^\star}{\A_s},\Delta_0^0\}$, the overall estimate trivially holds  for $\TT_\ell=\TT_0$.
We thus have derived that
\begin{align*}
 \#\TT_\ell - \#\TT_0 + 1 
 &\lesssim  \max\{\norm{u^\star}{\A_s},\Delta_0^0\}^{1/s} \, \sum_{n = 0}^\ell (\Delta_n^0)^{-1/s}\\
 &\le \max\{\norm{u^\star}{\A_s},\Delta_0^0\}^{1/s} \!\! \sum_{\substack{(\ell',k') \in \QQ \\ (\ell',k') \le (\ell,k)}} (\Delta_{\ell'}^{k'})^{-1/s}
 \quad \text{for all $(\ell,k) \in \QQ$},
\end{align*}
where the hidden constant depends only on $\Cstab$, $\qred$, $\Crel$, $\Cmesh$, $1-\lctr/\lopt$, $\Cmark$, $\Crel'$, $\Delta_0^0$, and $s$.
Finally, we employ linear convergence~\eqref{eq:theorem:linconv} to bound the latter sum by means of the geometric series
$$
\sum_{\substack{(\ell',k') \in \QQ \\ (\ell',k') \le (\ell,k)}} (\Delta_{\ell'}^{k'})^{-1/s}
\reff{eq:theorem:linconv}\le \Clin^{1/s} \, (\Delta_{\ell}^{k})^{-1/s} \, \sum_{\substack{(\ell',k') \in \QQ \\ (\ell',k') \le (\ell,k)}} 
	(\qlin^{1/s})^{|(\ell,k)| - |(\ell',k')|}
\le \frac{\Clin^{1/s}}{1 - \qlin^{1/s}} \, (\Delta_{\ell}^{k})^{-1/s}.
$$
Combining the latter two estimates, we see that
\begin{align}\label{eq:rates:step5}
 \#\TT_\ell - \#\TT_0 + 1 \lesssim \max\{\norm{u^\star}{\A_s},\Delta_0^0\}^{1/s} (\Delta_{\ell}^{k})^{-1/s}
 \quad \text{for all } (\ell,k) \in \QQ,
\end{align}
where the hidden constant depends only on $\Cstab$, $\qred$, $\Crel$, $\Cmark$, $1-\lctr/\lopt$, $\Cmark$, $\Crel'$, $\Clin$, $\qlin$, $\Delta_0^0$, and  $s$. 

{\bf Step~6.}
Let $(\ell',k') \in \QQ$. Together with Step~5, the geometric series proves that 
\begin{align*}
& \sum_{\substack{(\ell,k) \in \QQ \\ (\ell,k) \le (\ell',k')}} \#\TT_\ell
 \reff{eq:TT0}\le (\#\TT_0) \!\! \sum_{\substack{(\ell,k) \in \QQ \\ (\ell,k) \le (\ell',k')}} (\#\TT_\ell - \#\TT_0 + 1)
 \\& \quad
 \reff{eq:rates:step5} \lesssim \max\{\norm{u^\star}{\A_s},\Delta_0^0\}^{1/s} \!\!\!\! \sum_{\substack{(\ell,k) \in \QQ \\ (\ell,k) \le (\ell',k')}} (\Delta_\ell^k)^{-1/s}
\reff{eq:theorem:linconv} \le \frac{\Clin^{1/s}}{1 - \qlin^{1/s}} \, \max\{\norm{u^\star}{\A_s},\Delta_0^0\}^{1/s} \, (\Delta_{\ell'}^{k'})^{-1/s}.
\end{align*}
Rearranging this estimate, we end up with
\begin{align*}
 \sup_{(\ell',k') \in \QQ} \bigg( \sum_{\substack{(\ell,k) \in \QQ \\ (\ell,k) \le (\ell',k')}} \#\TT_\ell \bigg)^s \, \Delta_{\ell'}^{k'}
 \lesssim \max\{\norm{u^\star}{\A_s},\Delta_0^0\},
\end{align*}
where the hidden constant depends only on $\Cstab$, $\qred$, $\Crel$, $\Cmesh$, $1-\lctr/\lopt$, $\Cmark$, $\Crel'$, $\Clin$, $\qlin$, $\Delta_0^0$, $\#\TT_0$, and $s$.
This concludes the proof.
\end{proof}

\section{Numerical experiments}
\label{section:numerics}

\noindent 
This section provides numerical experiments that underpin our theoretical findings, where we employ $H^1$-conforming lowest-order FEM in 2D; see~\eqref{eq:nonlinear:Xh}. On the one hand, we present an example for AFEM with optimal PCG solver, cf.~Section~\ref{section:linear}, and on the other, an example for AFEM for a strongly monotone nonlinearity, cf.~Section~\ref{section:nonlinear}. 
For each problem, we compare the performance of Algorithm~\ref{algorithm} for
\begin{itemize}
\item different values of $\lambda \in \{1,10^{-1},10^{-2},10^{-3},10^{-4}\}$,
\item different values of $\theta \in \{0.1,0.3,0.5,0.7,0.9,1\}$,
\end{itemize}
where $\theta = 1$ corresponds to uniform mesh-refinement. 
In the experiments, the domain $\Omega\subset\R^2$ is either the $Z$--shaped domain from Figure~\ref{fig:geometries} (left) or the $L$--shaped domain from Figure~\ref{fig:geometries} (right).
For further examples of Algorithm~\ref{algorithm} applied to strongly monotone PDEs, we refer to the own work~\cite{banach} as well as the recent preprint~\cite{hw19}.

\begin{figure}
  \centering
  \raisebox{-0.5\height}{\includegraphics[width=0.4\textwidth]{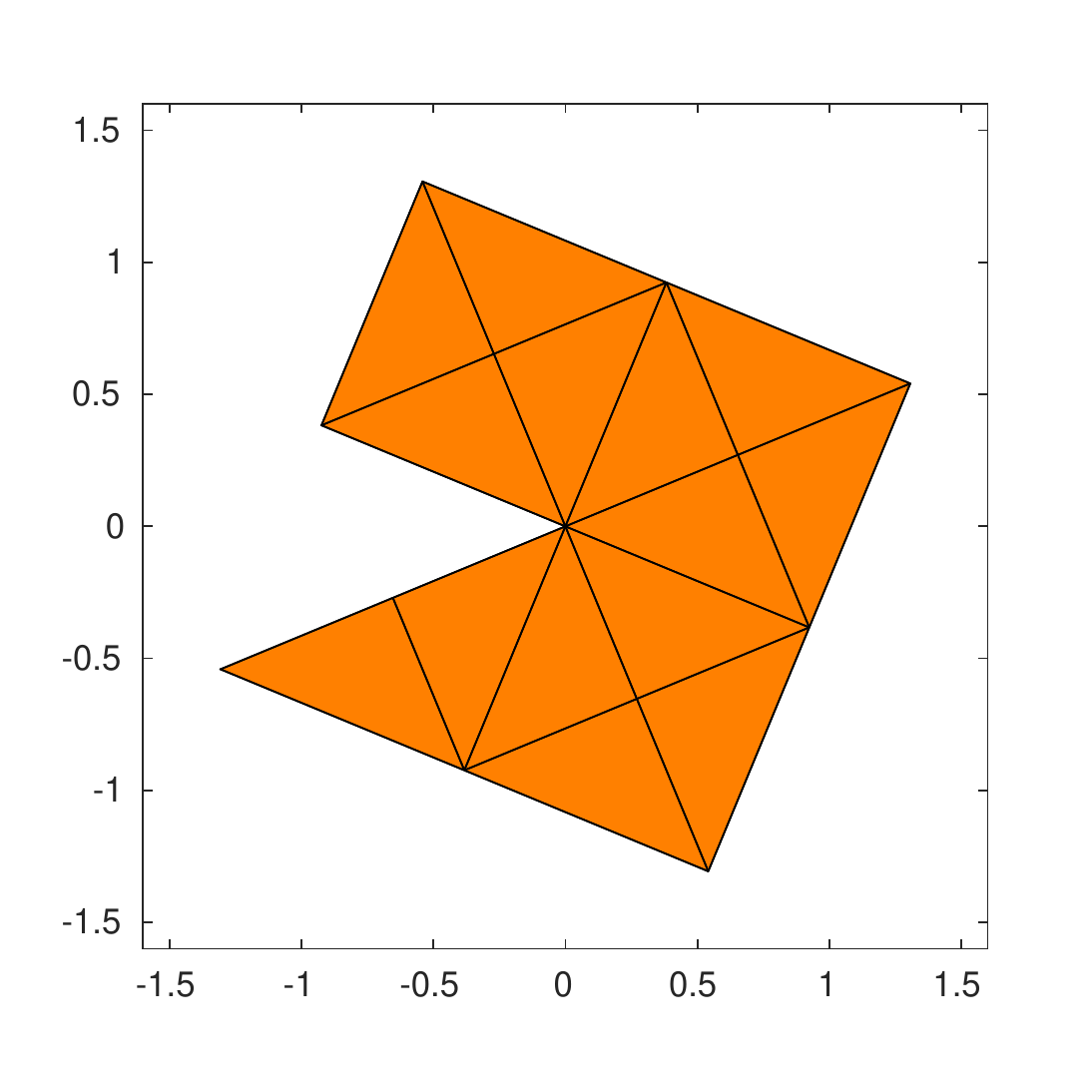}}
  \quad\quad
  \raisebox{-0.5\height}{\includegraphics[width=0.4\textwidth]{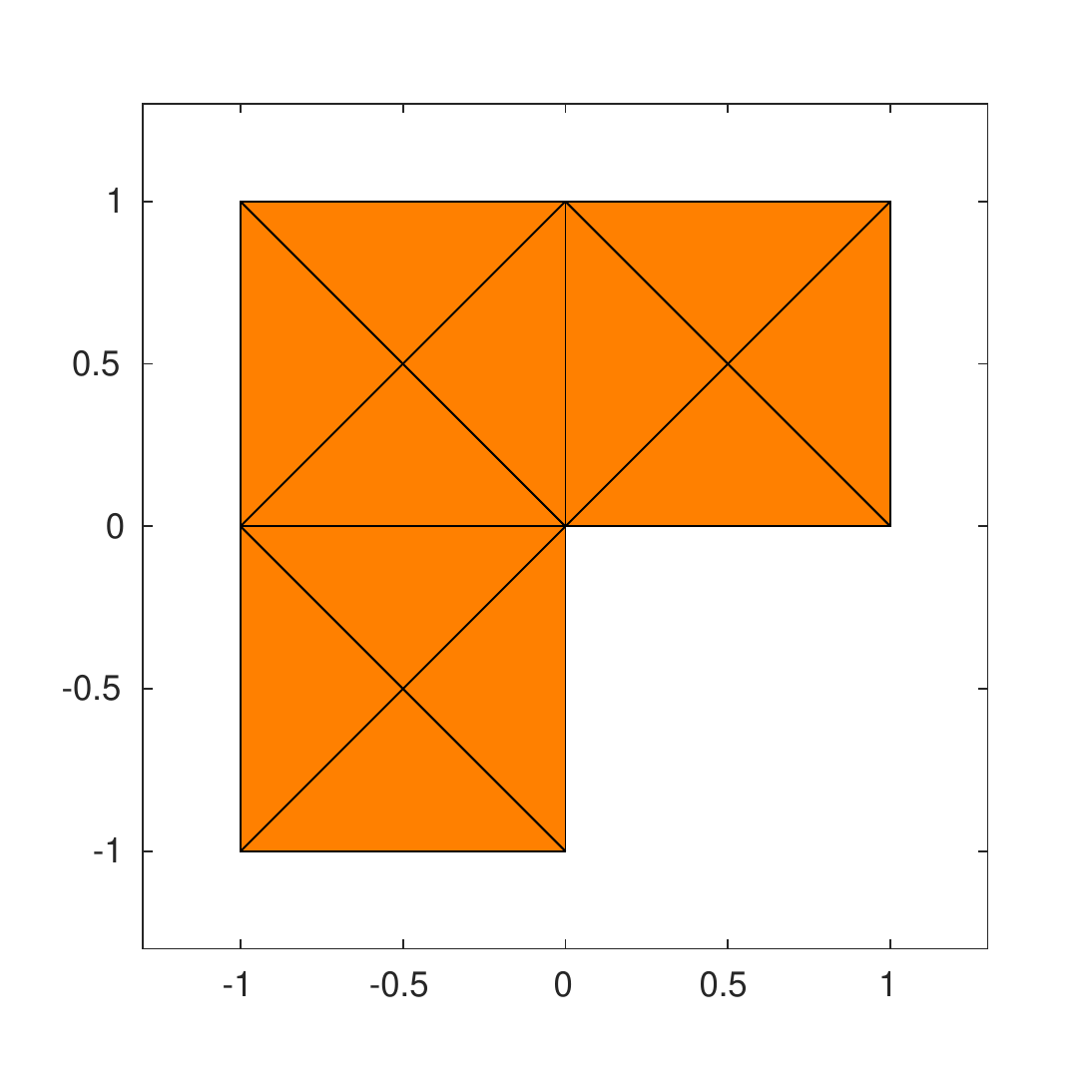}}
  \caption{$Z$--shaped domain $\Omega\subset\R^2$ with initial mesh $\TT_0$ (left) and $L$--shaped domain $\Omega\subset\R^2$ with initial mesh $\TT_0$ (right).}
  \label{fig:geometries}
\end{figure}

\subsection{AFEM for linear elliptic PDE with optimal PCG solver}
\label{section:linear_numerics}

\noindent We consider the following Poisson problem with homogeneous Dirichlet boundary conditions
\begin{align}
\begin{split}
	-\Delta u^\star&=1\quad\text{in }\Omega,\\
	u^\star&=0\quad\text{on }\Gamma,
	\end{split}
\end{align}
for both geometries from Figure~\ref{fig:geometries}. As an optimal preconditioner for the PCG solver, we use the multilevel additive Schwarz preconditioner of~\cite[Section 7.4.1]{dissFuehrer}.

\begin{figure}
	\centering
	\raisebox{-0.5\height}{\includegraphics[width=0.49\textwidth]{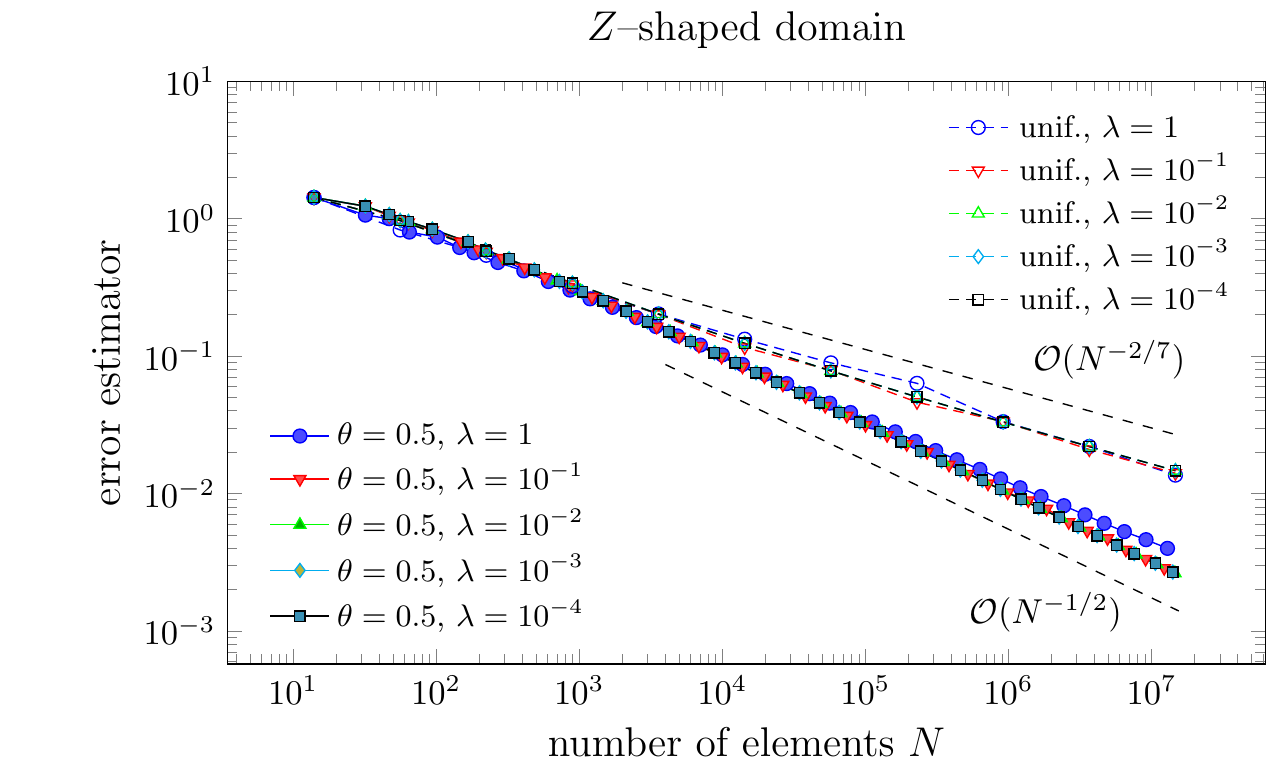}}
	\hfill
	\raisebox{-0.5\height}{\includegraphics[width=0.49\textwidth]{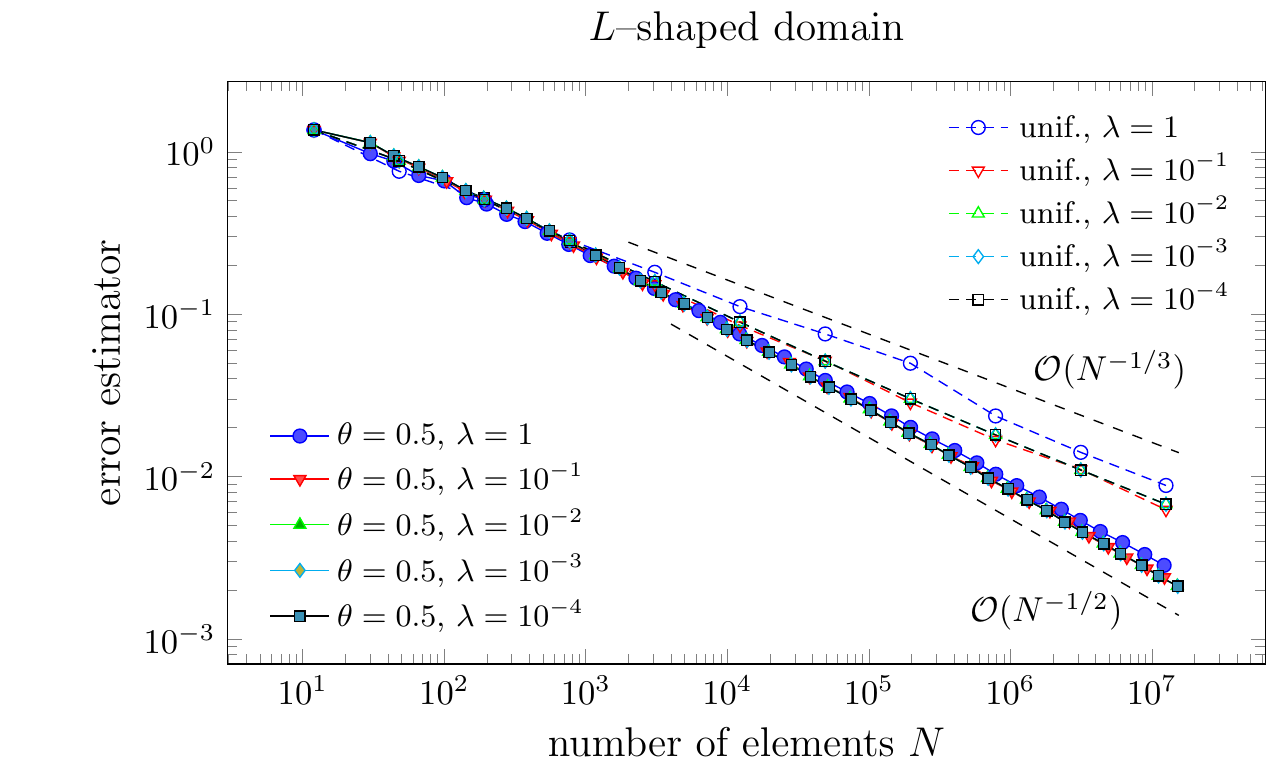}}\vspace{0.2cm}
	\raisebox{-0.5\height}{\includegraphics[width=0.49\textwidth]{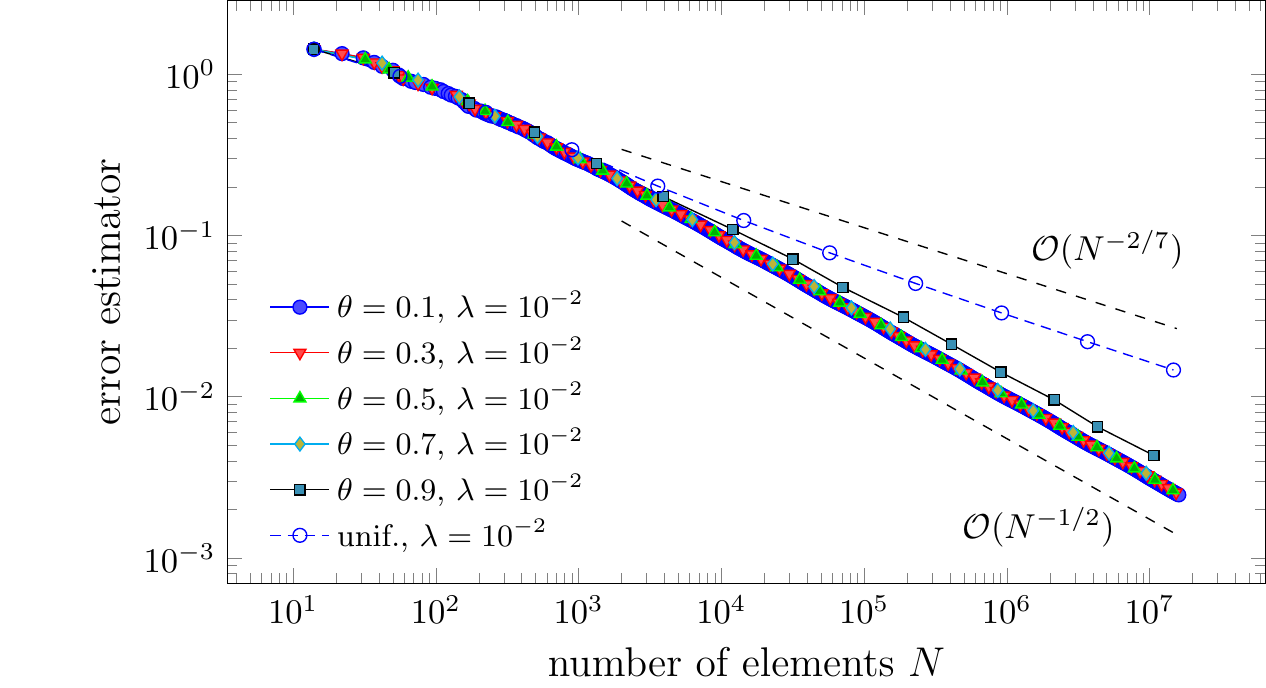}}
	\hfill
	\raisebox{-0.5\height}{\includegraphics[width=0.49\textwidth]{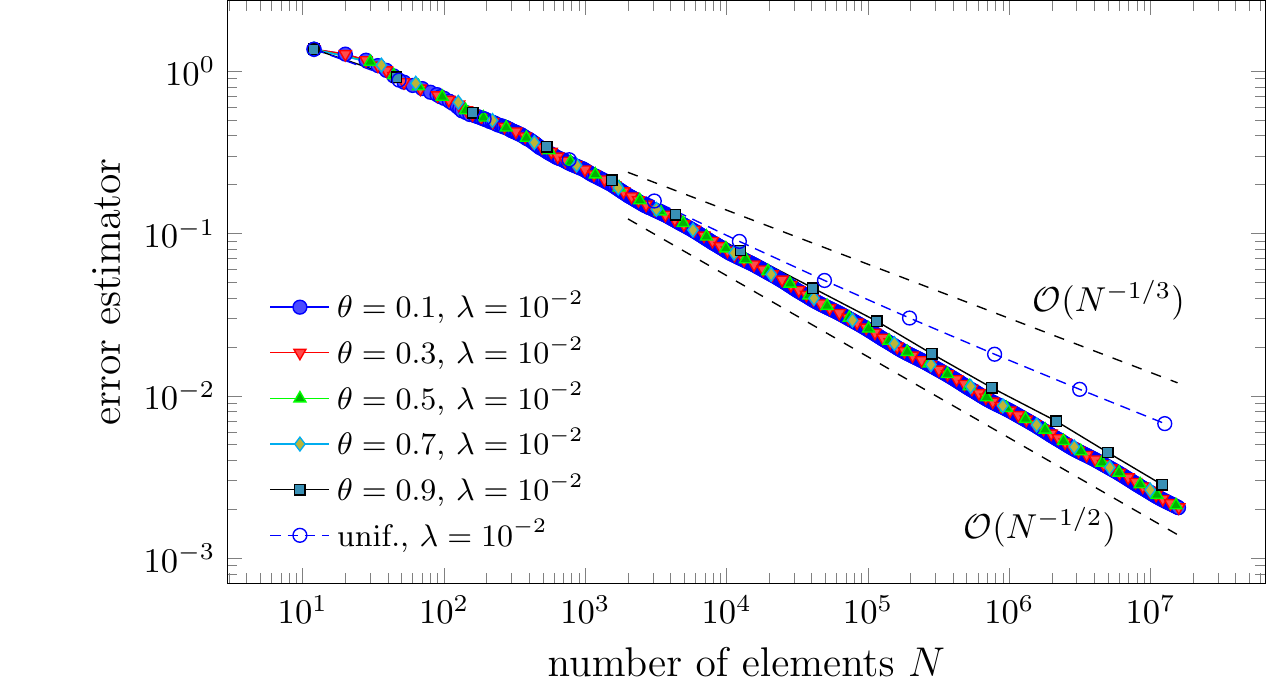}}
	\caption{Example from Section~\ref{section:linear_numerics}: Error estimator $\eta_\ell(u_\ell^\k)$ of the last step of the PCG solver with respect to the number of elements $N$ for $\theta=0.5$ and $\lambda\in\{1,10^{-1},\ldots,10^{-4}\}$ (top) as well as for $\lambda=10^{-2}$ and $\theta\in\{0.1, 0.3, \ldots, 0.9\}$ (bottom).}
\label{fig:conv_linear}
\end{figure}

In Figure~\ref{fig:conv_linear}, we compare Algorithm~\ref{algorithm} for different values of $\theta$ and $\lambda$, and uniform mesh-refinement. To this end, the error estimator $\eta_\ell(u_\ell^\k)$ of the last step of the PCG solver is plotted over the number of elements. Recall that  $\eta_\ell(u_\ell^\k)\simeq\Delta_\ell^\k$ according to Proposition~\ref{lemma:reliable}. We see that uniform mesh-refinement leads to the suboptimal rate of convergence $\OO(N^{-2/7})$ for the $Z$--shaped domain and $\OO(N^{-1/3})$ for the $L$--shaped domain.
 Algorithm~\ref{algorithm} regains the optimal rate of convergence $\OO(N^{-1/2})$, which empirically confirms Theorem~\ref{theorem:rates}. The latter rate of convergence appears to be even robust with respect to $\theta\in\{0.1,0.3, \ldots,0.9\}$ as well as $\lambda\in\{1, 10^{-1},\ldots,10^{-4}\}$.

\begin{figure}
	\centering
	\raisebox{-0.5\height}{\includegraphics[width=0.49\textwidth]{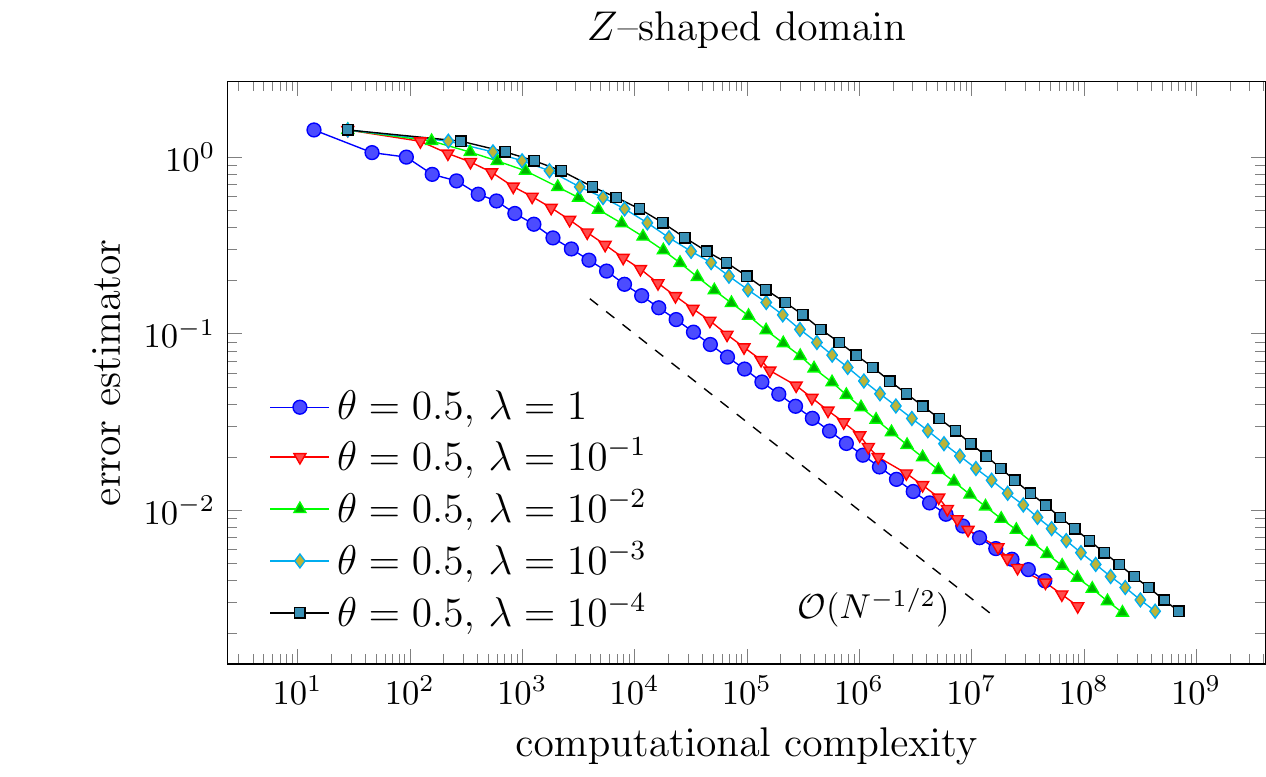}}
	\hfill
	\raisebox{-0.5\height}{\includegraphics[width=0.49\textwidth]{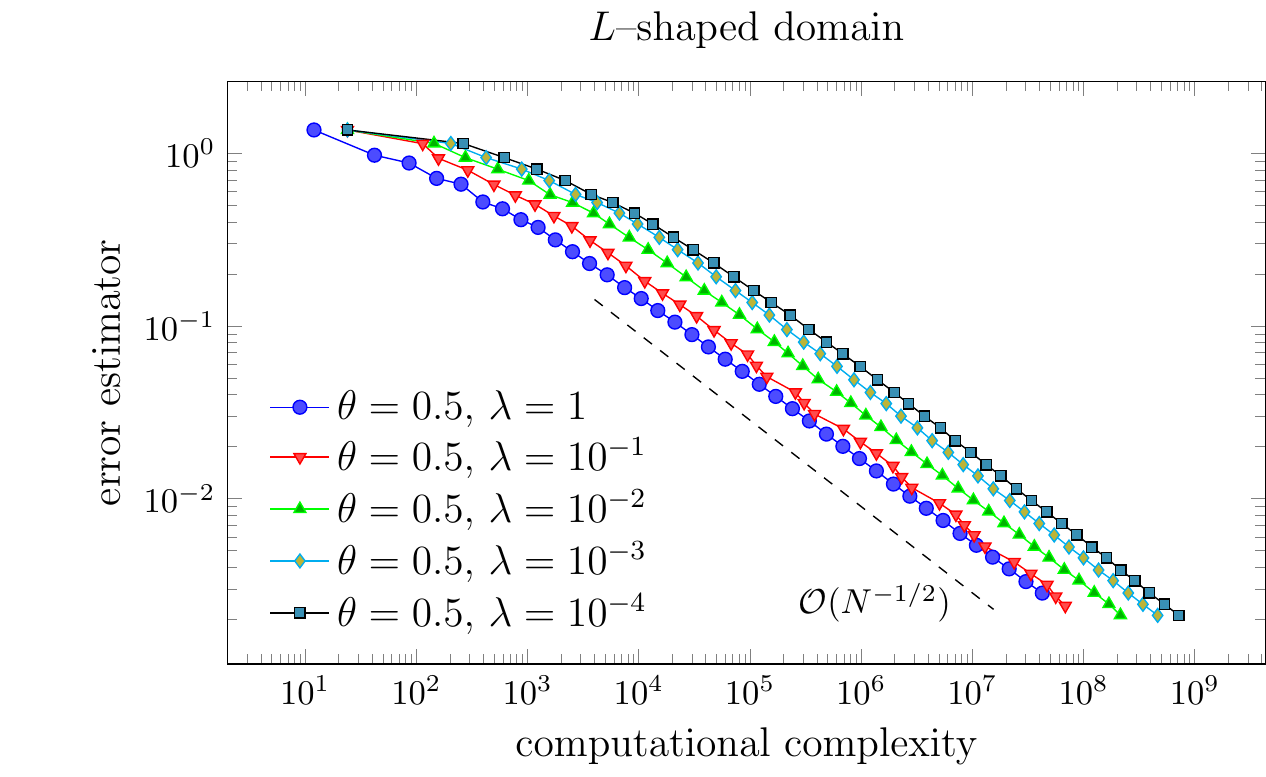}}\vspace{0.2cm}
	\raisebox{-0.5\height}{\includegraphics[width=0.49\textwidth]{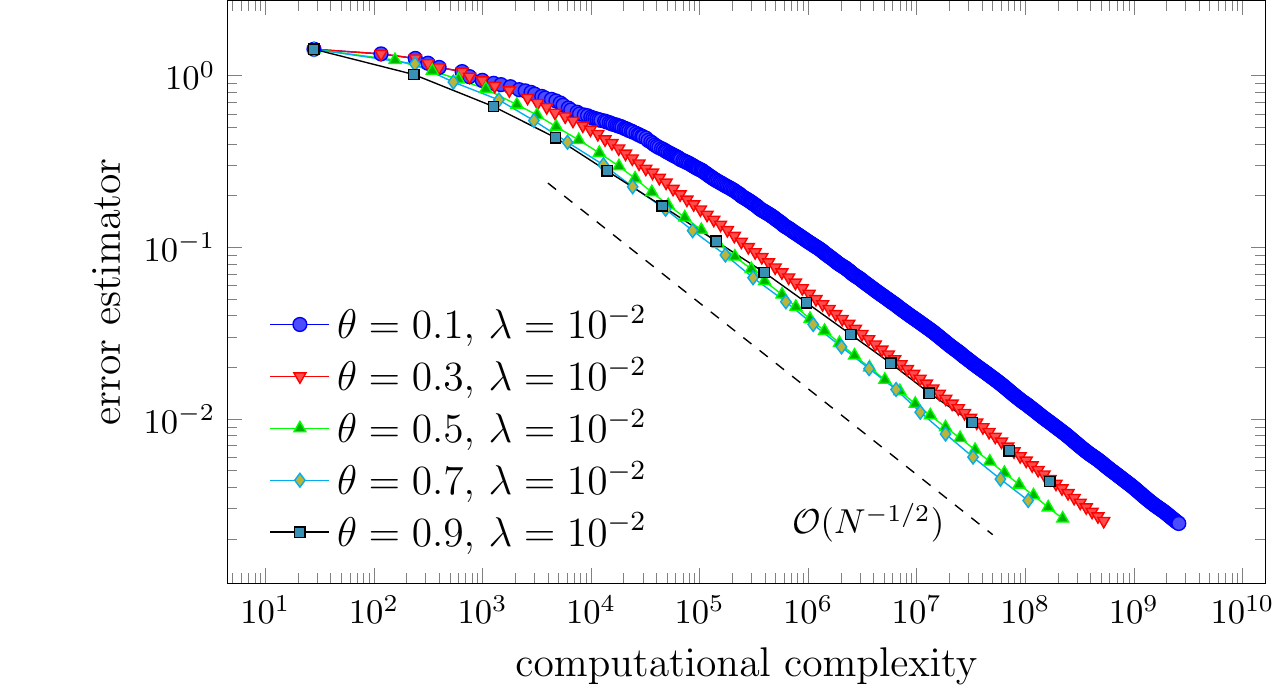}}
	\hfill
	\raisebox{-0.5\height}{\includegraphics[width=0.49\textwidth]{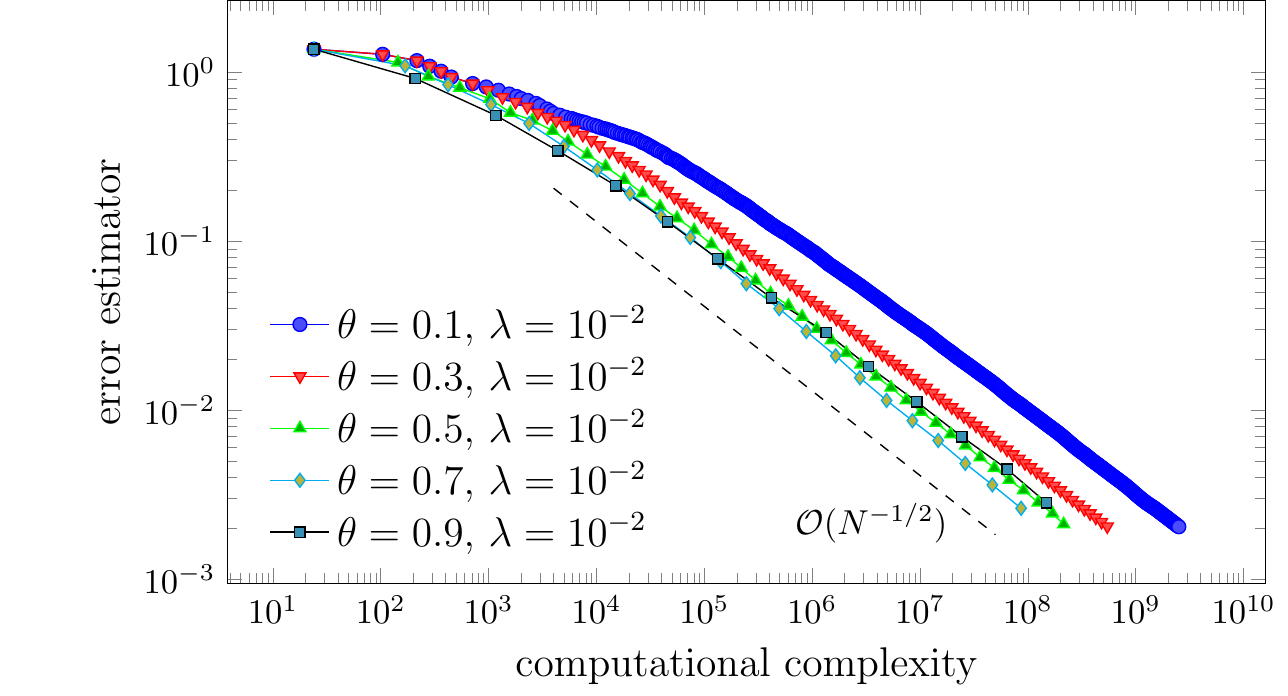}}
	\caption{Example from Section~\ref{section:linear_numerics}: Error estimator $\eta_\ell(u_\ell^\k)$ of the last step of the PCG solver with respect to the cumulative sum $\sum_{(\ell',k')\le(\ell,\k)}\#\TT_{\ell'}$ for $\theta=0.5$ and $\lambda\in\{1,10^{-1},\ldots,10^{-4}\}$ (top) as well as for $\lambda=10^{-2}$ and $\theta\in\{0.1, 0.3, \ldots, 0.9\}$ (bottom).}
\label{fig:compl_linear}
\end{figure}

In Figure~\ref{fig:compl_linear}, we aim to underpin that Algorithm~\ref{algorithm} has the optimal rate of convergence with respect to the computational complexity. To this end, we plot the error estimator $\eta_\ell(u_\ell^\k)$ of the last step of the PCG solver over the cumulative sum $\sum_{(\ell',k')\le(\ell,\k)}\#\TT_{\ell'}$.
In accordance with Theorem~\ref{theorem:rates}, we observe again the optimal order $\OO(N^{-1/2})$.

\begin{figure}
	\centering
	\raisebox{-0.5\height}{\includegraphics[width=0.49\textwidth]{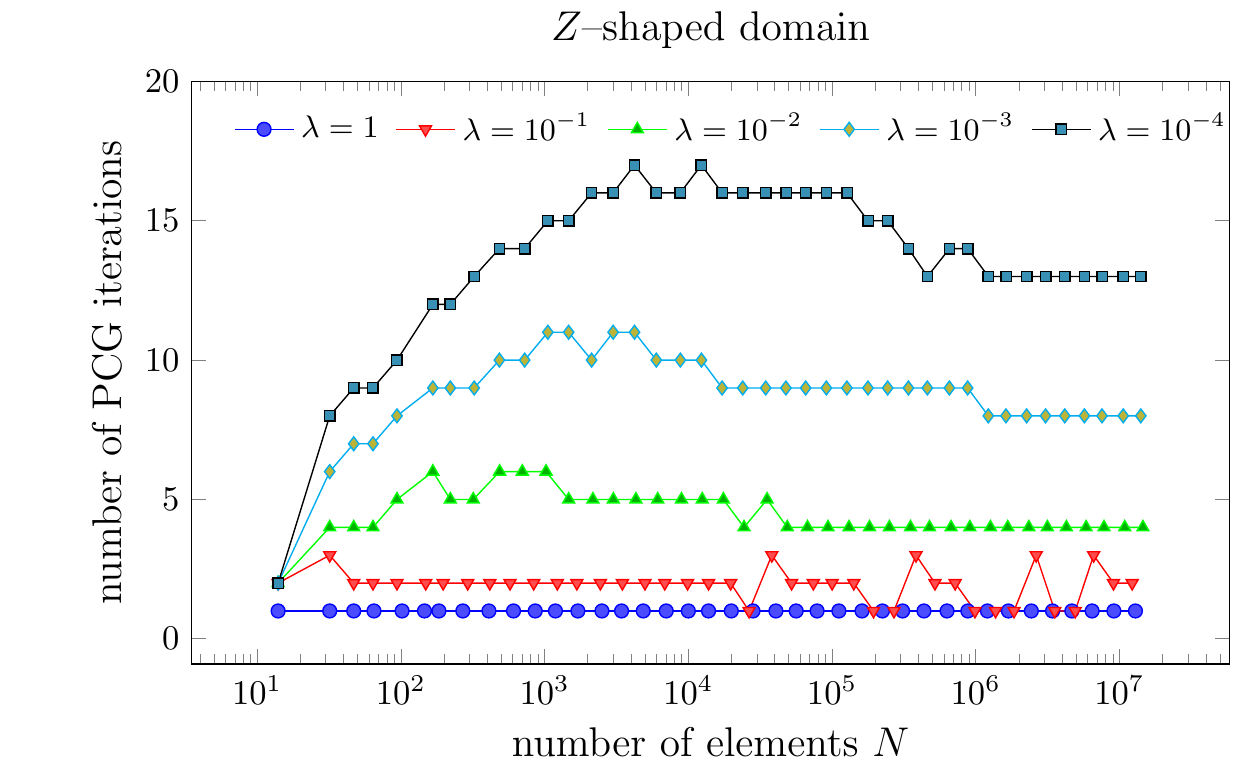}}
	\hfill
	\raisebox{-0.5\height}{\includegraphics[width=0.49\textwidth]{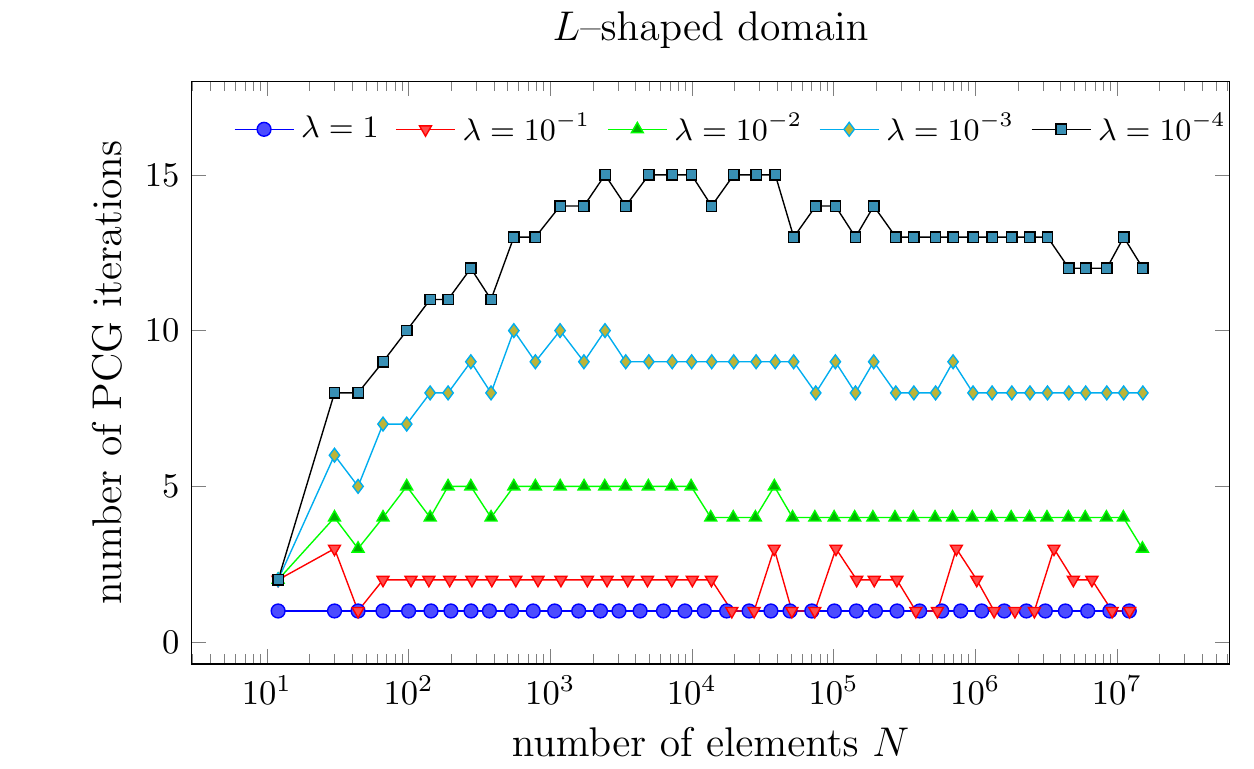}}\vspace{0.2cm}
	\raisebox{-0.5\height}{\includegraphics[width=0.49\textwidth]{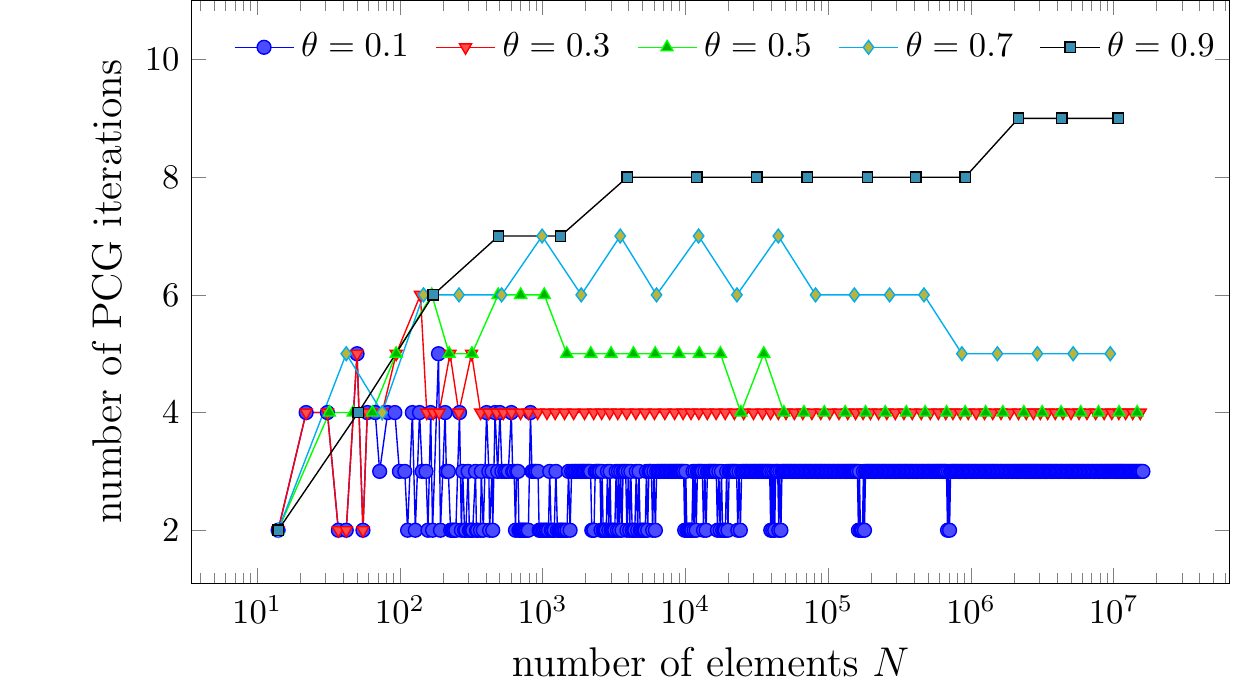}}
	\hfill
	\raisebox{-0.5\height}{\includegraphics[width=0.49\textwidth]{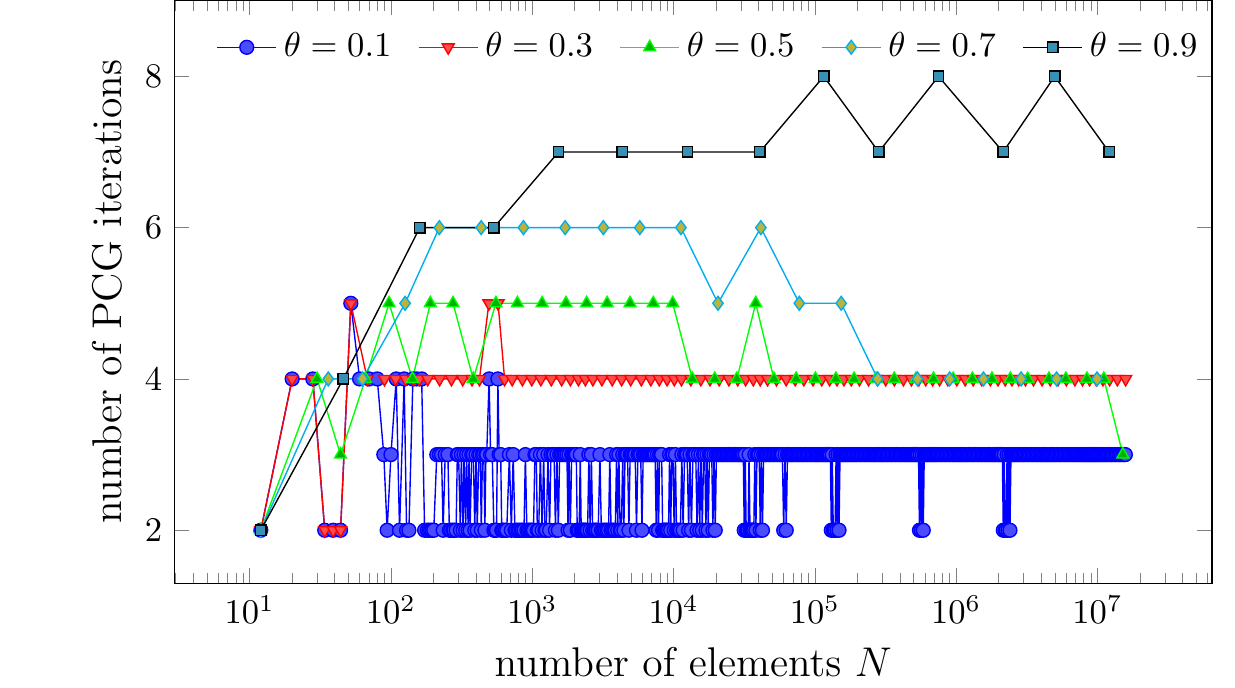}}
	\caption{Example from Section~\ref{section:linear_numerics}: Number of PCG iterations with respect to the number of elements $N$ for $\theta=0.5$ and $\lambda\in\{1,10^{-1},\ldots,10^{-4}\}$ (top) as well as for $\lambda=10^{-2}$ and $\theta\in\{0.1, 0.3, \ldots, 0.9\}$ (bottom).}
\label{fig:nIt_linear}
\end{figure}

In Figure~\ref{fig:nIt_linear}, we take a look at the number of PCG iterations.
We observe that a larger value of $\lambda$ or a smaller value of $\theta$ lead to a smaller number of PCG iterations. Nonetheless, in each case, this number stays uniformly bounded. 

\subsection{AFEM for strongly monotone nonlinearity}
\label{section:nonlinear_numerics}

\noindent We consider the 
problem
\begin{align}
\begin{split}
  -\div \big({a}(\cdot,|\nabla u^\star|^2)\nabla u^\star\big) &= 1 \quad \text{in } \Omega,\\
 u^\star &= 0 \quad \text{on } \Gamma,
\end{split}
\end{align}
where the scalar nonlinearity ${a}: \Omega \times \R_{\geq 0} \rightarrow \R$ is defined by
\begin{align}
	a(x,t):=1 + \frac{\ln(1+t)}{1+t}.
\end{align}
Then,~\eqref{item:G_bounded}--\eqref{item:G_lipschitz1} hold with $\alpha=c_a'\approx 0.9582898017 $ and $L=C_a' \approx 1.542343818$.

\begin{figure}
	\centering
	\raisebox{-0.5\height}{\includegraphics[width=0.49\textwidth]{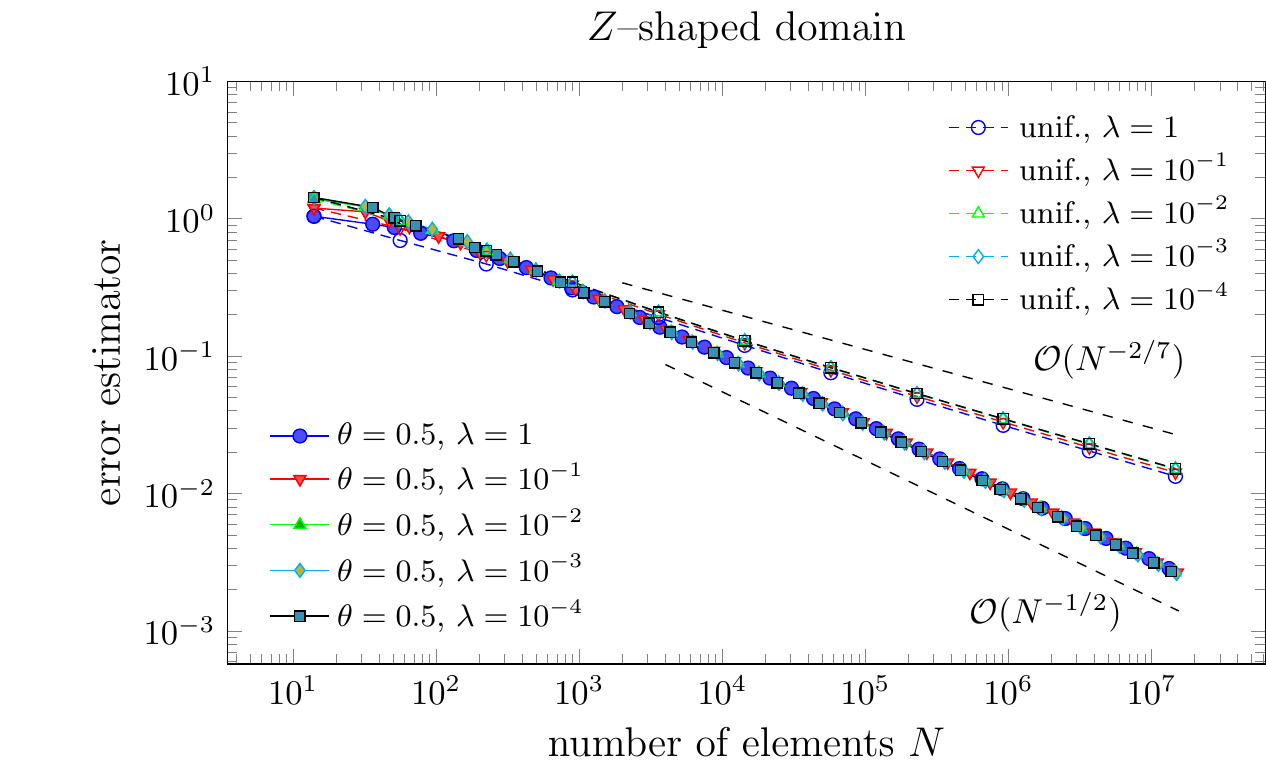}}
	\hfill
	\raisebox{-0.5\height}{\includegraphics[width=0.49\textwidth]{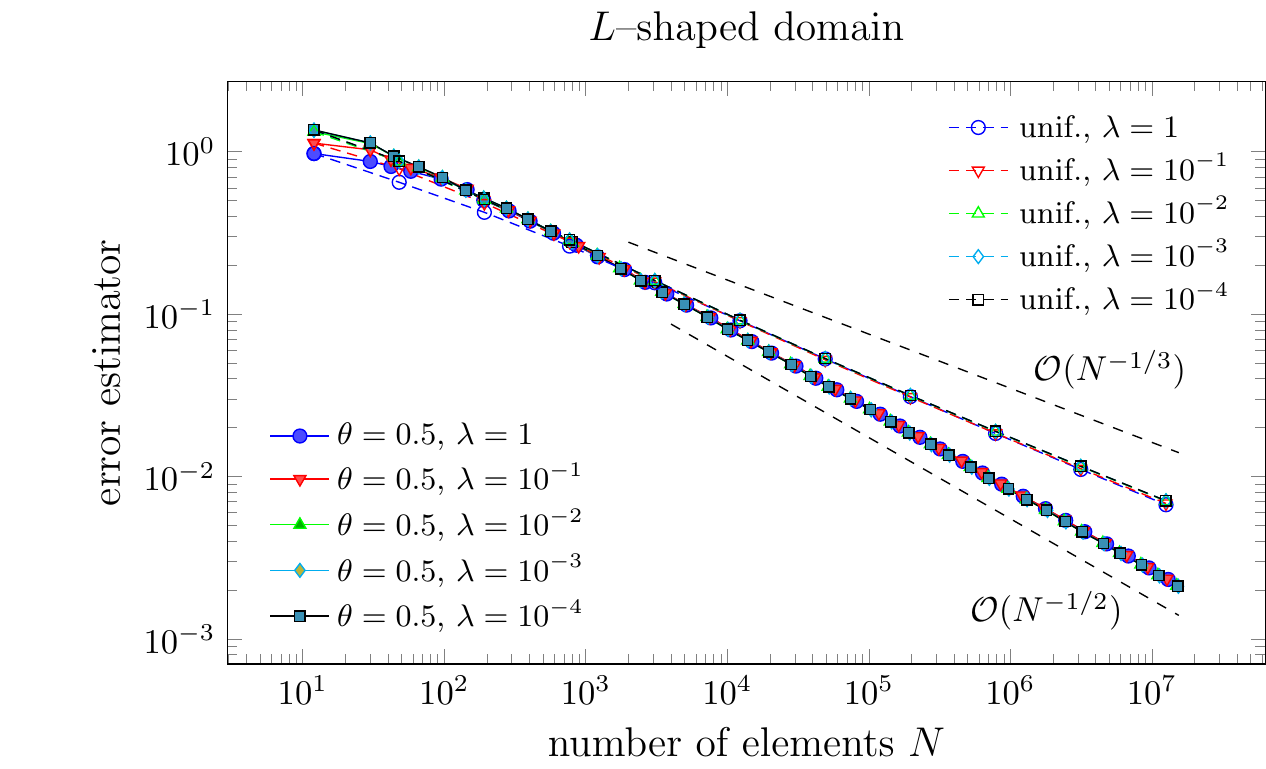}}\vspace{0.2cm}
	\raisebox{-0.5\height}{\includegraphics[width=0.49\textwidth]{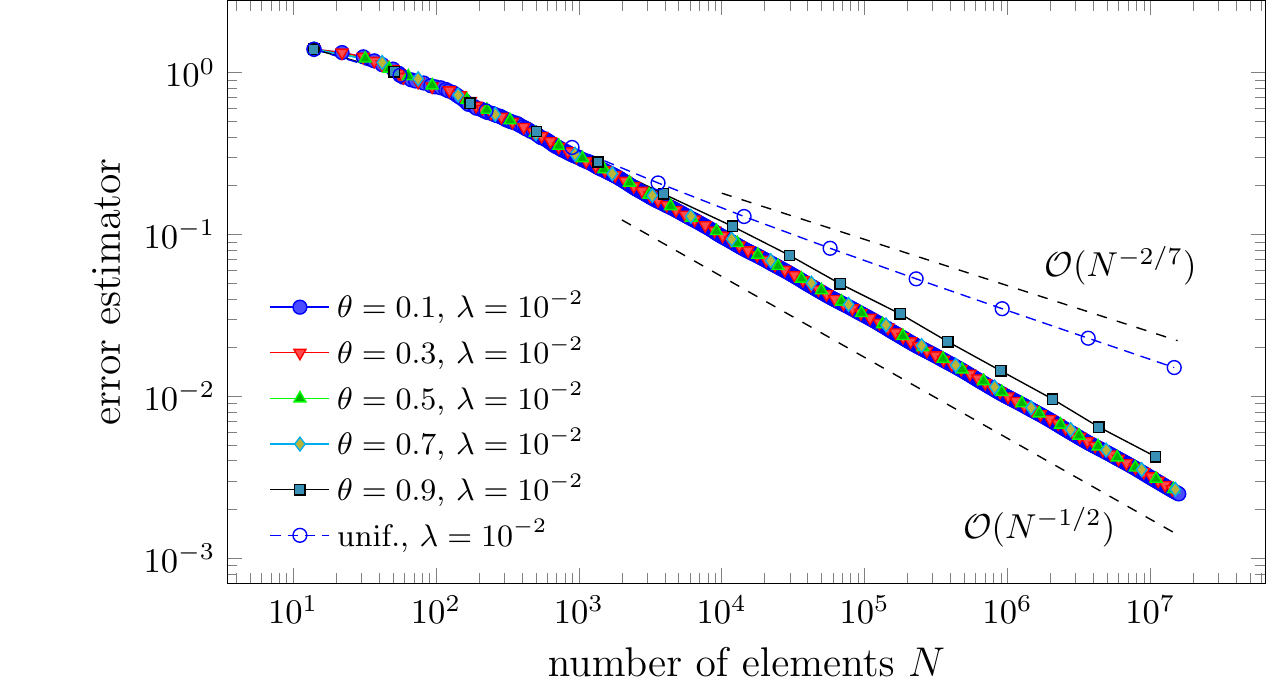}}
	\hfill
	\raisebox{-0.5\height}{\includegraphics[width=0.49\textwidth]{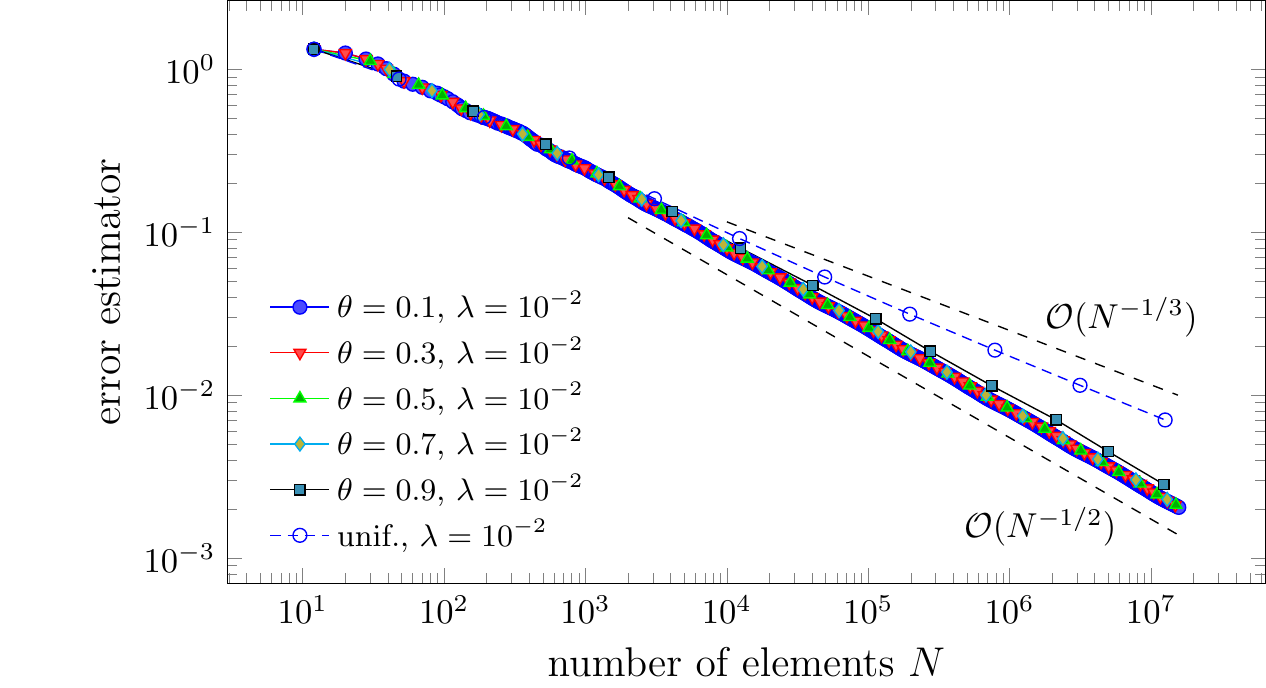}}
	\caption{Example from Section~\ref{section:nonlinear_numerics}: Error estimator $\eta_\ell(u_\ell^\k)$ of the last step of the Picard iteration with respect to the number of elements $N$ for $\theta=0.5$ and $\lambda\in\{1,10^{-1},\ldots,10^{-4}\}$ (top) as well as for $\lambda=10^{-2}$ and $\theta\in\{0.1, 0.3, \ldots, 0.9\}$ (bottom).}
\label{fig:conv_nonlinear}
\end{figure}

In Figure~\ref{fig:conv_nonlinear}, we compare Algorithm~\ref{algorithm} for different values of $\theta$ and $\lambda$, and uniform mesh-refinement. To this end, the error estimator $\eta_\ell(u_\ell^\k)$ of the last step of the Picard iteration is plotted over the number of elements. 
We see that uniform mesh-refinement leads to the suboptimal rate of convergence $\OO(N^{-2/7})$ for the $Z$--shaped domain and $\OO(N^{-1/3})$ for the $L$--shaped domain. Algorithm~\ref{algorithm} regains the optimal rate of convergence $\OO(N^{-1/2})$, independently of the actual choice of $\theta\in\{0.1,0.3, \ldots,0.9\}$ and $\lambda\in\{1, 10^{-1},\ldots,10^{-4}\}$ for both geometries.
Since $\eta_\ell(u_\ell^\k)\simeq\Delta_\ell^\k$, this again empirically confirms Theorem~\ref{theorem:rates}.

\begin{figure}
	\centering
	\raisebox{-0.5\height}{\includegraphics[width=0.49\textwidth]{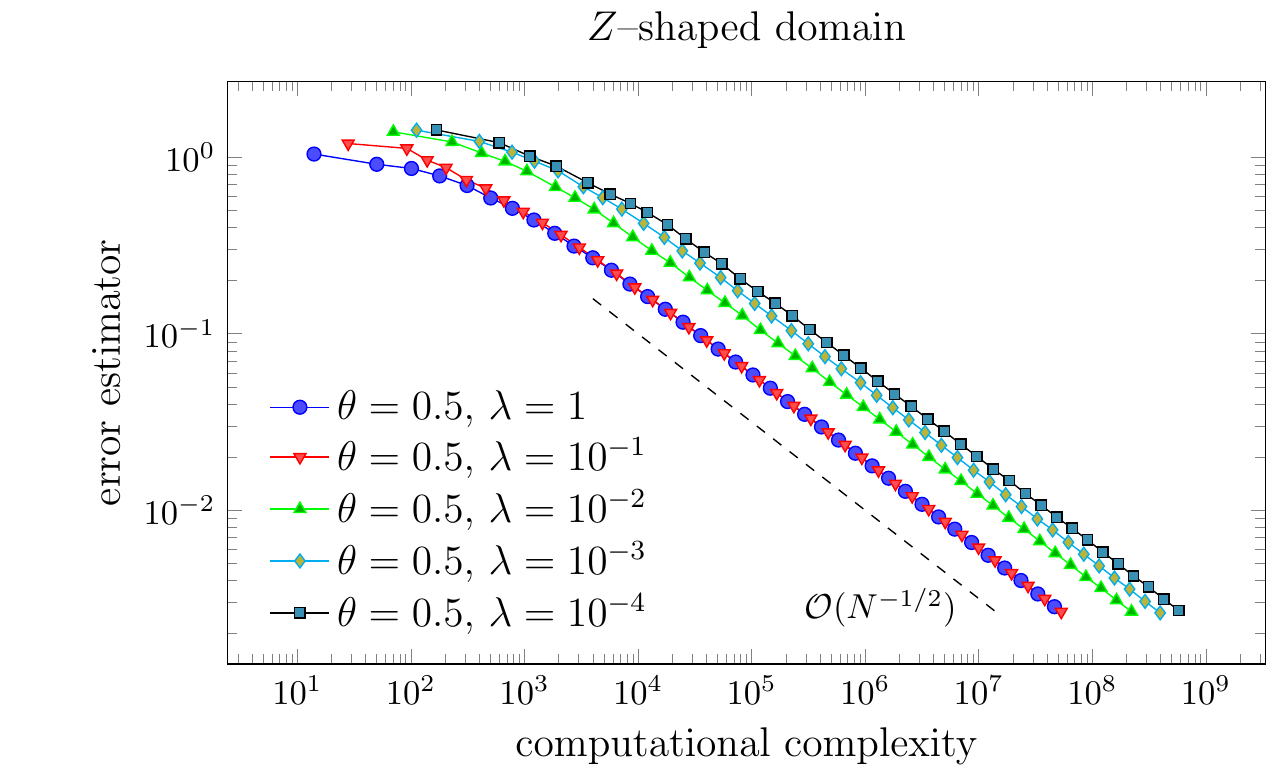}}
	\hfill
	\raisebox{-0.5\height}{\includegraphics[width=0.49\textwidth]{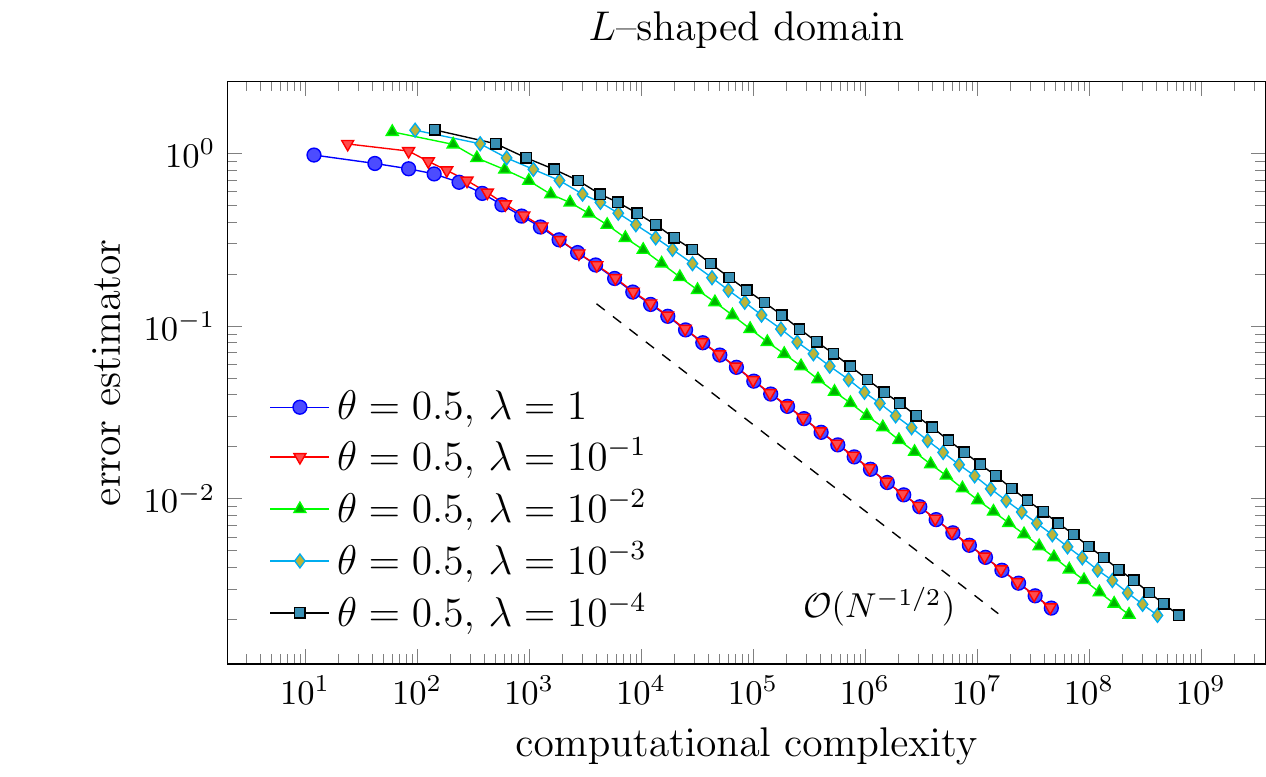}}\vspace{0.2cm}
	\raisebox{-0.5\height}{\includegraphics[width=0.49\textwidth]{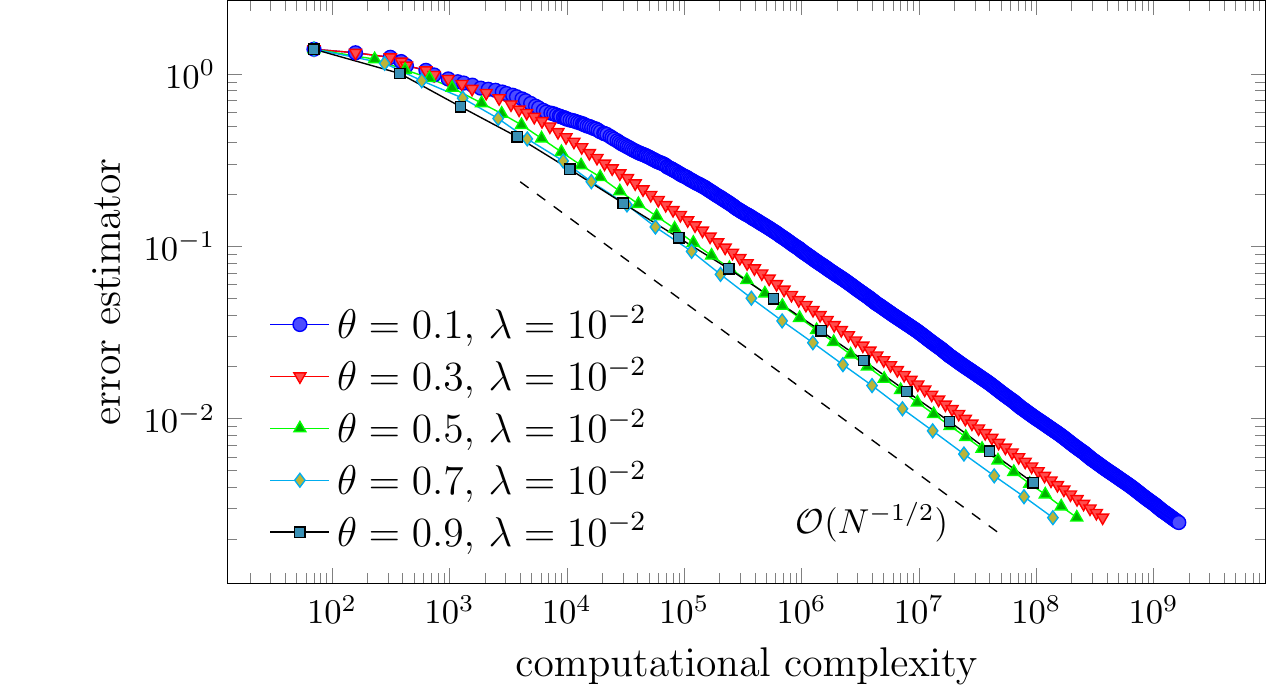}}
	\hfill
	\raisebox{-0.5\height}{\includegraphics[width=0.49\textwidth]{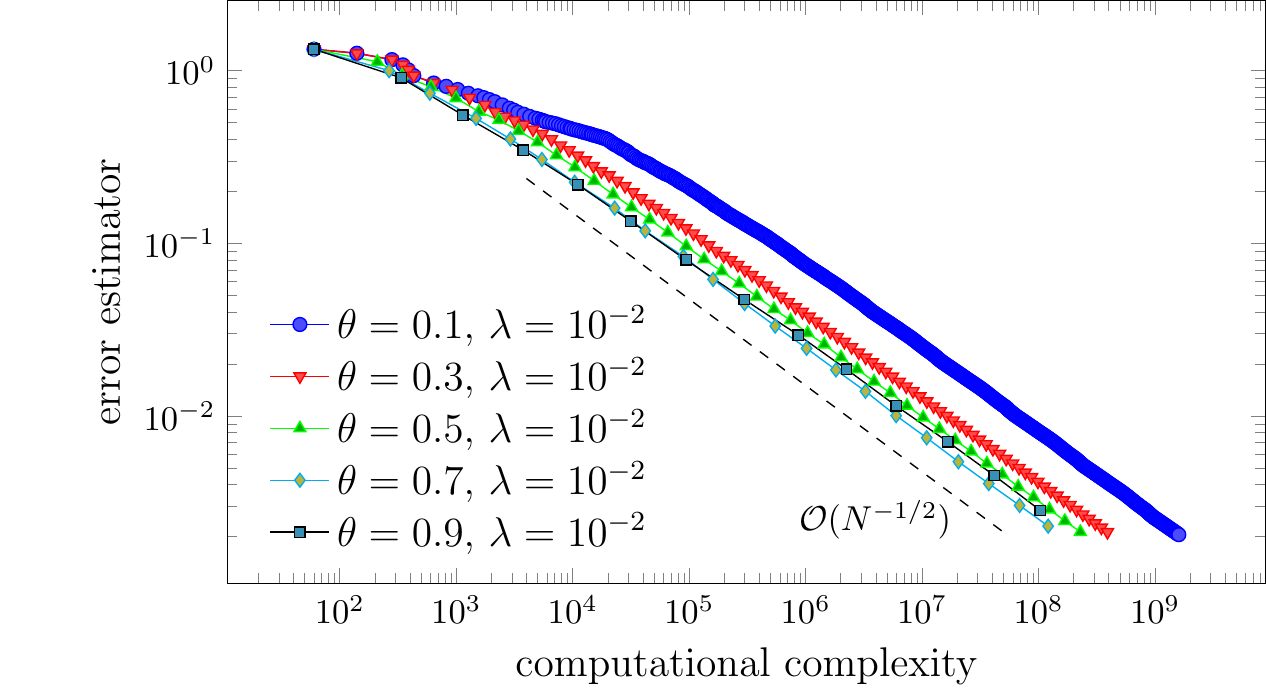}}
	\caption{Example from Section~\ref{section:nonlinear_numerics}: Error estimator $\eta_\ell(u_\ell^\k)$ of the last step of the Zarantonello iteration with respect to the cumulative sum $\sum_{(\ell',k')\le(\ell,\k)}\#\TT_{\ell'}$ for $\theta=0.5$ and $\lambda\in\{1,10^{-1},\ldots,10^{-4}\}$ (top) as well as for $\lambda=10^{-2}$ and $\theta\in\{0.1, 0.3, \ldots, 0.9\}$ (bottom).}
\label{fig:compl_nonlinear}
\end{figure}

In Figure~\ref{fig:compl_nonlinear}, 
we plot the estimator $\eta_\ell(u_\ell^\k)$  of the last step of the Zarantonello iteration over the cumulative sum $\sum_{(\ell',k')\le(\ell,\k)}\#\TT_{\ell'}$.
As predicted in Theorem~\ref{theorem:rates}, we observe that Algorithm~\ref{algorithm} regains the optimal order of convergence $\OO(N^{-1/2})$ with respect to the computational complexity.
The rate seems to be independent of the values of $\lambda$ or $\theta$.

We mention that the number of Zarantonello iterations (not displayed) behaves similarly as the number of PCG iterations of Figure~\ref{fig:nIt_linear}.
A larger value of $\lambda$ or a smaller value of $\theta$ lead to less iterations, where the number stays uniformly bounded in each case.



\bibliographystyle{alpha}
\bibliography{literature}


\end{document}